\documentclass[12pt]{article}
\usepackage[top=1.2in,bottom=1.2in,left=1.2in,right=1.2in,marginpar=1in]{geometry}

\usepackage{graphicx} % Required for inserting images

\usepackage{latexsym,amsfonts,amssymb,amsmath,amsthm}
\usepackage{comment}
\usepackage{pinlabel}
\usepackage[linktocpage=true]{hyperref}
\hypersetup{
  colorlinks   = true, %Colours links instead of ugly boxes
  urlcolor     = blue, %Colour for external hyperlinks
  linkcolor    = blue, %Colour of internal links
  citecolor   = red %Colour of citations
}

\usepackage{tikz}
\usepackage{todonotes}
\usepackage{cleveref}
\usepackage{tikz-cd}
\usepackage{mathabx}

\usepackage{subcaption}

 \usepackage{fbox}

\makeatletter
\let\@wraptoccontribs\wraptoccontribs
\makeatother

\newcommand{\nc}{\newcommand}
\nc{\dmo}{\DeclareMathOperator}
\nc{\nt}{\newtheorem}

\newtheorem{thm}{Theorem}[section]
\newtheorem{cor}[thm]{Corollary}
\newtheorem{lem}[thm]{Lemma}
\newtheorem{prop}[thm]{Proposition}

\newtheorem{ques}[thm]{Question}

\newtheorem*{fact*}{Fact}
\newtheorem{fact}[thm]{Fact}

\newtheorem*{namedtheorem}{\theoremname}
\newcommand{\theoremname}{testing}
\newenvironment{named}[1]{\renewcommand{\theoremname}{#1}\begin{namedtheorem}}{\end{namedtheorem}}

\theoremstyle{definition}
\newtheorem{defn}[thm]{Definition}

\newtheorem*{conv*}{Convention}
\theoremstyle{remark}
\newtheorem{rem}[thm]{Remark}

\nc{\Z}{\mathbb{Z}} %integers
\nc{\R}{\mathbb{R}} %reals
\nc{\Q}{\mathbb{Q}} %rationals
\nc{\N}{\mathbb{N}} %naturals
\nc{\set}[2]{\{#1 \,|\, #2\} }

\nc{\gp}{\Gamma\mathcal{G}} % graph product
\nc{\CX}{\mathcal{C} P(\gp)}
\nc{\G}{\Gamma}
\nc{\Star}{\mathsf{Star}}
\nc{\Link}{\mathsf{Link}}
\nc{\vol}{\mathsf{Vol}}
\nc{\eqw}{=_\mathsf{w}}
\nc{\eqe}{=_\mathsf{g}}
\nc{\Cay}{\text{Cay}}
\nc{\id}{\mathsf{id}}
\nc{\vkd}{\mathcal{D}_\mathsf{vK}} % van Kampen diagram
\nc{\dvkd}{\Delta} % dual van Kampen diagram
\nc{\dd}{\mathcal{D}} % disk diagram
\nc{\mH}{\mathsf{m}_H} % longest prism length of H generators

\newcommand{\Rfaktor}[2]{{\raisebox{.3em}{$#1$}\bigg/\raisebox{-.4em}{$#2$}}}

\newcounter{scomments}

 \newcounter{jcomments}

\newcounter{mtcomments}

 \newcounter{acomments}

 \newcounter{mccomments}

\makeatletter
\let\@wraptoccontribs\wraptoccontribs
\makeatother

\title{Stable subgroups of graph products}
\author{Sahana H Balasubramanya, Marissa Chesser, Alice Kerr, \\ 
Johanna Mangahas, Marie Trin}
\begin{document}

\maketitle

%%% Abstract
\begin{abstract}
    We extend the characterization of stable subgroups of right-angled Artin groups from \cite{RAAGstable} to the case of graph products of infinite groups. Specifically, we show that the stable subgroups of such graph products are exactly the subgroups that quasi-isometrically embed in the associated contact graph. Equivalently, they are the subgroups that satisfy a condition arising from the defining graph: a stable subgroup is an almost join-free subgroup. In particular, we generalize the equivalence between stable and purely loxodromic subgroups from \cite{RAAGstable} in the case where all torsion subgroups of the vertex groups are finite, and the equivalence between stable and infinite index Morse subgroups from \cite{Tran2019} in the case where the defining graph is connected.
\end{abstract}

%%%%%%%%%%%%%%%%%%%%%%%%%%%%%%%%%%%%

\tableofcontents

\section{Introduction}\label{sec:intro}

Given a simplicial graph $\Gamma$ and a family of finitely generated groups $\mathcal{G}=\{ G_v \mid v\in V(\Gamma) \}$ indexed by the vertices of $\Gamma$, the associated graph product $\gp$ is generated by the vertex groups $G_v$, with the additional relation that elements of adjacent vertex groups commute (as detailed in \Cref{sec:graphproducts}). Graph products were introduced by Green in her PhD thesis \cite{greennormalform} and have since been widely studied. They provide an interplay between free products and direct products of groups, and generalize both right-angled Artin groups (RAAGs) and right-angled Coxeter groups (RACGs). In this paper, we focus on the case where the defining graph $\Gamma$ has no isolated vertices, and where the groups $G_v$ are all finitely generated and infinite; this situation generalizes the one of RAAGs.

The goal of this paper is to characterize the \emph{stable} subgroups of such graph products.  A finitely generated subgroup $H$ of a finitely generated group $G$ is stable if its inclusion in $G$ is a quasi-isometric embedding and quasi-geodesics from $G$ with common endpoints in $H$ fellow-travel (see \Cref{def:stable} for details). In particular, they are hyperbolic subgroups, where the hyperbolicity is inherited from the ambient geometry of $G$.

This notion of stability was introduced in \cite{DT} to generalize both the quasi-convex subgroups of hyperbolic groups, and the convex cocompact subgroups of mapping class groups \cite{FarbMosher}. They have since been well studied for a broad collection of groups; see for example \cite{sistoHypEmbed,CordesDurham,AMST,HKim,abbott2022acylindrically}. In particular, \cite{RAAGstable} proved a geometric characterization of the stable subgroups of RAAGs.

\begin{thm} \emph{\cite[Theo 1.1]{RAAGstable}}
\label{thm:StableInRAAG}
Let $\Gamma$ be a finite connected graph and $H$ a finitely generated subgroup of the right angled Artin group $A(\Gamma)$. The following are equivalent:
\begin{enumerate}
    \item\label{thm-part:RAAG-stable} $H$ is stable in $A(\Gamma)$.
    \item\label{thm-part:RAAG-qi} The orbit maps of $H$ into the extension graph are quasi-isometric embeddings.
    \item\label{thm-part:RAAG-ploxo} $H$ is purely loxodromic for the action on the extension graph.
\end{enumerate}
\end{thm}

In this theorem stability is characterized by the action of the subgroup on the extension graph of $A(\Gamma)$. This space was introduced by Kim and Koberda in \cite{KimKoberda}, and is a hyperbolic graph on which $A(\Gamma)$ acts by isometries \cite{KimKoberda2}. In the language of \cite{BCKMT}, the equivalence of (\ref{thm-part:RAAG-stable}) and (\ref{thm-part:RAAG-qi}) in \Cref{thm:StableInRAAG} implies that the extension graph is a \emph{universal recognizing space} for the stable subgroups of $A(\Gamma)$. A universal recognizing space for a group $G$ is a hyperbolic space on which $G$ acts by isometries, and in which all the stable subgroups embed quasi-isometrically through their orbit maps for this action. With the additional (converse) assumption that every subgroup quasi-isometrically embedded through its orbit maps is stable, this has also been called a \textit{stability recognizing space} \cite{CDEK}.

A universal recognizing space for a group can be expected to play a role analogous to the one of the curve complex for mapping class groups. In fact, \Cref{thm:StableInRAAG} is also true for mapping class groups when the extension graph is replaced by the curve complex and with the additional assumption in condition (\ref{thm-part:RAAG-ploxo}) that the group $H$ is undistorted \cite{MR2465691,Hamenstaedt2005WordHE,BBKL}. More generally, conditions (\ref{thm-part:RAAG-stable}) and (\ref{thm-part:RAAG-qi}) are equivalent for hierarchically hyperbolic groups acting on the hyperbolic space associated to the maximal domain of some specific structure on the group \cite{ABD}, as well as for CAT(0) groups acting on the curtain model \cite{CAT0}, Morse-dichotomous groups acting on the contraction space \cite{Zbinden,PetytZalloum}, and certain relatively hyperbolic groups acting on the coned-off Cayley graph \cite{ADT}.

For graph products the analog of the curve graph that we are interested in is the \emph{contact graph of the prism-complex}; one can refer to \cite{Oyakawa} or \cite{genevois2017} for other possibilities. Our main theorem is as follows.

\begin{thm}
\label{thm-intro:recognizing-stable}
    Let $\Gamma$ be a finite simple graph with no isolated vertices, $\mathcal{G}=\set{G_v}{v\in V(\Gamma)}$ be a collection of finitely generated infinite groups, and $P(\gp )$ be the prism complex of the graph product $\gp$. Let $H$ be a finitely generated subgroup of the graph product $\gp$. Then the following are equivalent.
    \begin{enumerate}
        \item\label{thm-part-intro:recognizing-stable-stable} $H$ is stable in $\gp$. 
        \item\label{thm-part-intro:recognizing-stable-qi} The orbit maps of $H$ into the contact graph $\CX$ are quasi-isometric embeddings.
        \item\label{thm-part-intro:recognizing-stable-ploxo} $H$ is almost join-free.
    \end{enumerate}
\end{thm}

The contact graph $\CX$ is known to be a quasi-tree, and hence a hyperbolic space \cite[Corollary C]{valiunas2021graphproducts}.  It is therefore a universal recognizing space for $\gp$ under the conditions of \Cref{thm-intro:recognizing-stable}.

\begin{rem}
    The equivalence of (1) and (2) in \Cref{thm-intro:recognizing-stable} has also been proved independently (and simultaneously) by Joshua Perlmutter using different methods \cite[Corollary 5.4 and Proposition 5.5]{Perlmutter}, namely via a broader result for relatively hierarchically hyperbolic spaces \cite[Theorem 1.3]{Perlmutter}.
\end{rem}

We can relate the final equivalent condition in \Cref{thm-intro:recognizing-stable} to the last condition in Theorem \ref{thm:StableInRAAG}. A subgroup of a graph product is called a \emph{join subgroup} if it decomposes as a direct product of subgroups generated by disjoint subgraphs of the defining graph. \Cref{thm-intro:recognizing-stable} says the stable subgroups are exactly the \emph{almost join-free subgroups}, meaning those admitting only finite intersections with conjugates of join subgroups. (Note, our definition of almost join-free subgroups differs slightly from the one given by Tran in \cite{Tran}). In the case of RAAGs, the purely loxodromic condition is known to be equivalent to asking $H$ to be (almost) join-free \cite{KimKoberda2}, and this equivalence is a key part of the proof of \Cref{thm:StableInRAAG}.  This characterization cannot extend to general graph products due to the potential existence of infinite torsion subgroups, which are automatically purely loxodromic, but never stable (\Cref{sec:purely-loxodromic} details why stable implies purely loxodromic, but not the other way around). Under the added assumption that such subgroups do not exist, we recover the final equivalent statement of \Cref{thm:StableInRAAG}.

\begin{thm}\label{introthm:main-torsion-free}
    Let $\Gamma$ be a finite simple graph with no isolated vertices, $\mathcal{G}=\set{G_v}{v\in V(\Gamma)}$ be a collection of infinite finitely generated groups with no infinite torsion subgroups, and $P(\gp )$ be the prism complex of the graph product $\gp$. A finitely generated subgroup $H$ of $\gp$ is stable if and only if it is purely loxodromic with respect to the action on the contact graph $\CX$.
\end{thm}

In the non-hyperbolic cases of both $A(\Gamma)$ for connected $\Gamma$ and mapping class groups, there is another characterization of stable subgroups: they are exactly the infinite-index Morse subgroups \cite{Tran2019,HKim}. This is not the case for RACGs \cite[Theorem 1.11 and Corollary 1.12]{Tran2019}, however it does hold for graph products of infinite groups, when the defining graph is connected.

\begin{thm}\label{thm-intro:infindex}
    Let $\Gamma$ be a finite connected simple graph with at least two vertices, and $\mathcal{G}=\set{G_v}{v\in V(\Gamma)}$ be a collection of finitely generated infinite groups. A subgroup $H$ of the graph product $\gp$ is stable if and only if it is infinite index and Morse.
\end{thm}

The hypotheses on the vertex groups and on the graphs in the above theorems are necessary. For \Cref{thm-intro:infindex}, if $\G$ is not connected, then $\gp$ is a free product of infinite groups, with each of these free factors being both infinite index and Morse. At the same time, those free factors will be stable if and only if they are hyperbolic. \Cref{thm-intro:recognizing-stable} and \Cref{introthm:main-torsion-free} assume a slightly weaker hypothesis, namely that there are no isolated vertices in $\Gamma$.  If $\Gamma$ contains an isolated vertex, then the associated vertex group will always be almost join-free, while its orbit in the contact graph will have finite diameter. This vertex group is also a free factor, and therefore Morse, so as before it will be stable if and only if it is hyperbolic. Moreover, it was shown in \cite[Corollary D]{valiunas2021graphproducts} that, under the conditions of \Cref{thm-intro:recognizing-stable}, the action of $\gp$ on the contact graph is the largest acylindrical action. When isolated vertices are allowed, the authors showed in \cite{BCKMT} that the largest acylindrical action of such a graph product does not necessarily provide a universal recognizing space for stable subgroups.

The other important hypothesis in the above theorems is that the vertex groups are infinite. Otherwise, we could consider the case where $\G$ is a path with two vertices, and the vertex groups are $\Z$ and $\Z_2$. Here $\gp$ is an infinite join group (so not almost join-free), while being stable and finite index within itself, and with $\CX$ having finite diameter. Note that this hypothesis on vertex groups excludes RACGs from our theorem; in that case an analogous version of the curve complex is given by the the graph of maximal product regions introduced by Oh \cite{Oh}, and under some hypotheses Cashen, Dani, Edletzberger, and Karrer prove the equivalence between (\ref{thm-part-intro:recognizing-stable-stable}) and (\ref{thm-part-intro:recognizing-stable-qi}) for this graph \cite[Cor 6.10]{CDEK}. As RACGs are hierarchically hyperbolic groups, the work in \cite{ABD} also applies. It is therefore natural to ask if a more general result exists for graph products.

\begin{ques}
    Does a characterization of stable subgroups similar to \Cref{thm-intro:recognizing-stable} hold when the vertex groups are allowed to be finite?
\end{ques}

Our final observation concerns what stability implies about subgroup type.  Stable subgroups of RAAGs are free, by \Cref{thm:StableInRAAG} and \cite[Prop 8.1, 8.2]{RAAGstable}. For graph products, we get the following generalization as a direct consequence of \Cref{thm-intro:recognizing-stable}, the contact graph being a quasi-tree, and the fact that groups which quasi-isometrically embed into quasi-trees are virtually free (see for example \cite[Corollary 10.7]{Button2023}).

\begin{cor} \label{cor:virtually free}
    Let $\Gamma$ be a finite simple graph with no isolated vertices, and $\mathcal{G}=\set{G_v}{v\in V(\Gamma)}$ be a collection of finitely generated infinite groups. The stable subgroups of the graph product $\gp$ are virtually free.
\end{cor}

\noindent\textbf{Organization of the paper.} \Cref{sec:background} introduces the concepts and notation we use throughout the paper. \Cref{sec:DD} is dedicated to the study of disk diagrams, which are our main technical tool. Their study serves in the proof of (\ref{thm-part-intro:recognizing-stable-ploxo})$\Longrightarrow$(\ref{thm-part-intro:recognizing-stable-qi}) of \Cref{thm-intro:recognizing-stable}. \Cref{sec:main_proof} gives the proof of \Cref{thm-intro:recognizing-stable}; the diagram below describes how this theorem is proved for the convenience of the reader. Finally, in \Cref{sec:purely-loxodromic} and \Cref{sec:morse}, we study the link between stable, purely loxodromic, and Morse subgroups of graph products.

\[
\begin{tikzcd}[column sep=5cm, row sep=large]
1) H \text{ is stable in } \gp
  \arrow[dd, Rightarrow, "\Cref{prop:InfJoinNotStable}"] 
& \parbox{10cm}{2) The orbit of $H$ into the\\ contact graph $\CX$ is a\\ quasi-isometric embedding} 
\arrow[l, Rightarrow, "\Cref{prop:qi-ilplies-stable}"]
\\
\\
3) H \text{ is almost join-free} 
  \arrow[r, Rightarrow, "\Cref{prop:quasi-isometric embedding}", shift left=1.2ex] 
  \arrow[r, Rightarrow, "\Cref{thm:join-busting}"', shift right=1.2ex] 
& \parbox{10cm}{$(H, d_{H})$ is q.i.e into $(\gp, d_{p})$ \\ $H$ is $N$ join-busting}
  \arrow[uu, Rightarrow, "\Cref{thm:ploxo-impies-qi}", shift left=12ex]
\end{tikzcd}
\]
\noindent\textbf{Acknowledgments.} The authors are grateful to the organizers of the 2023 Women in Groups, Geometry, and Dynamics program, where our collaboration started. We would like to thank Anthony Genevois for his feedback, both at the beginning of this project and on a later draft, and for his enlightening comments on \cite{Genevois2019}. Our thanks also to Chris Cashen and Annette Karrer for their explanation of the case of right-angled Coxeter groups.

This work was supported by funding from the EPSRC grant EP/V027360/1 ``Coarse geometry of groups and spaces'' via the third author, the Simons Foundation (965204, JM) via the fourth author, and MPI MiS via the third and fifth authors.

%%%%%%%%%%%%%%%%%%%%%%%%%%%%%%%%%%%%
\section{Preliminaries}\label{sec:background}
 
%%%%%%%%%%%%%%%%%%%%%%%%%%%%%%%%%%%%

\subsection{Coarse geometry}

Two metric spaces $(X,d_X)$ and $(Y,d_Y)$ are said to be \emph{quasi-isometric} if there is a map $f : X \to Y$ and two real numbers $\lambda \geq 1$ and $c\geq 0$ such that 
\begin{itemize}
   \item$\forall x,x'\in X \quad \frac{1}{\lambda}d_X(x,x')-c \leq d_Y(f(x),f(x'))  \leq  \lambda d_X(x,x')+c$
   \item$\forall y\in Y, \exists x\in X \quad d_Y(y,f(x))\leq c$.
\end{itemize}
The map $f$ is called a \emph{$(\lambda,c)$-quasi-isometry}; the first condition alone makes the map $f$ a \emph{$(\lambda,c)$--quasi-isometric embedding}. A \emph{quasi-isometry} (resp. \emph{quasi-isometric embedding}) is a map which is a $(\lambda,c)$--quasi-isometry (resp. quasi-isometric embedding) for some pair $(\lambda,c)$. A quasi-isometric embedding from an interval of $\R$ or $\Z$ is called a \emph{quasi-geodesic}, which is a \emph{geodesic} when $\lambda=1$ and $c=0$.

\subsection{Stable subgroups}

\label{def:stable}

Let $G$ be a finitely generated group. A finitely generated subgroup $H \leq G$ is said to be \emph{stable} if both of the following conditions hold:\begin{itemize}
    \item $H$ is undistorted in $G$, i.e.~the inclusion $H\hookrightarrow G$ is a quasi-isometric embedding. 
    \item For each choice of constants $\lambda\geq 1, c\geq0$, there is a constant $M \geqslant0$ (depending on $\lambda, c$) such that any two $(\lambda,c)$--quasi-geodesics in $G$ with the same end points in $H$ are each contained in the $M$--neighborhood of the other.
\end{itemize} 

Note that stability is independent of the choice of finite generating sets for $H$ and $G$ \cite{DT}.  In particular, it is preserved under group automorphisms such as conjugation.  We note this explicitly because much of the work in this paper is dedicated to the study of condition (3) in \Cref{thm-intro:recognizing-stable} which deals with conjugacy. Namely, we have:

\begin{lem} \label{lem:ConjStableIsStable} Let $H$ be a stable subgroup of a finitely generated group $G$. For any $g\in G$ the subgroup $g Hg^{-1}$ is also a stable subgroup of $G$.
\end{lem}

\subsection{Graph products}\label{sec:graphproducts} 

Let $\Gamma = (V(\Gamma),E(\Gamma))$ be a graph, where the set $V(\Gamma)$ is the set of vertices of the graph and the set $E(\Gamma) \subseteq V(\Gamma) \times V(\Gamma)$ is the set of (non-oriented) edges of the graph. Let $\mathcal{G}= \{ G_v | v\in V(\Gamma) \}$ be a family of groups called \emph{vertex groups}. In this paper we will always assume that $\G$ is a finite simple graph and that the vertex groups are all non-trivial. The \emph{graph product} $\gp$ associated to the pair $(\Gamma,\mathcal{G})$ is defined by 
\begin{equation*}\label{graph product}
\gp := \Rfaktor{\bigg( \underset{v\in V(\G)}{\bigast}G_v\bigg)}{\big\langle\hspace{-1mm} \big\langle \thinspace [g,h]\ \mid \ g\in G_v, h\in G_u \text{ with } \{u,v\}\in E(\Gamma)\big\rangle\hspace{-1mm} \big\rangle}.
\end{equation*}
Notice that if all the vertex groups are isomorphic to $\mathbb{Z}$ then $\gp$ is a right-angled Artin group, and if the vertex groups are all isomorphic $\mathbb{Z}_2$, then $\gp$ is a right-angled Coxeter group.

For every (induced) subgraph $\G'$ of $\G$, the subgroup $G_{\G'}=\langle G_v\ |\ v\in V(\G')\rangle$ is known as a \emph{parabolic subgroup} of $\gp$. For ease of notation, we will typically not distinguish between the vertex set $V(\G')$ and the graph $\G'$.

\subsubsection{Join, star, and link subgroups}\label{subsubsec:join-star-link-subs}
Recall that a graph $\G$ is said to be a \emph{join} of the nonempty subgraphs $\G_1$ and $\G_2$ if $V(\G)=V(\G_1)\sqcup V(\G_2)$, and for every $u\in V(\G_1)$, $v\in V(\G_2)$, we have that $\{u,v\}\in E(\G)$. For any graph product associated to $\G$, this implies that $\gp=G_{\G_1}\times G_{\G_2}$. We will refer to any subgraph $\G'$ of $\G$ that splits as a join as a \emph{join subgraph} of $\G$, and the corresponding parabolic subgroup $G_{\G'}$ as a \emph{join subgroup} of $\gp$. Note that when talking about subgraphs we will always mean induced subgraphs.

For a graph $\G$, recall that the \emph{link} of the vertex $v\in V(\G)$, denoted $\Link(v)$, is the set of vertices adjacent to $v$, and the \emph{star} of $v$ is $v\cup\Link(v)$, which we denote by $\Star(v)$. We will refer to subgraphs induced by such sets of vertices as \emph{link subgraphs} and \emph{star subgraphs} of $\G$ respectively, and the corresponding parabolic subgroups $G_{\Link(v)}, G_{\Star(v)}$ as \emph{link subgroups} and \emph{star subgroups} of $\gp$.

We extend the notions of $\Link$ and $\Star$ to any subgraph $\G'$ of $\G$, so that $\Link(\G')=\bigcap_{v\in V(\G')}\Link(v)$ (that is, the vertices adjacent to every vertex in $\G'$) and $\Star(\G')= \G' \cup \Link(\G')$.  These definitions imply the following lemma.

\begin{lem}\label{lem:StarOfASubgraph}
    If $\G$ is a graph, and $\Lambda\subseteq\Lambda'\subseteq\G$ where $\Lambda'$ is a join, then $\Star(\Lambda)$ is a join.
\end{lem}

\begin{proof}
   Let $\Lambda'$ be the join of $\G_1$ and $\G_2$. Suppose $\Lambda\subseteq \G_1$, then $\emptyset\neq\G_2\subseteq\Link(\Lambda)$, and $\Star(\Lambda)$ is the join of $\Lambda$ with $\Link(\Lambda)$. The same holds if $\Lambda\subseteq \G_2$. Otherwise, $\Lambda$ is the join of $\G_1\cap\Lambda$ and $\G_2\cap\Lambda$, so $\Star(\Lambda)$ is a join whether or not $\Link(\Lambda)$ is nonempty.
\end{proof}

\begin{defn}
    A subgroup of a graph product is called \textit{join-free} (respectively \textit{star-free}) if it has trivial intersection with any conjugate of a join subgroup (respectively a star subgroup). A subgroup of a graph product is called \textit{almost join-free} (respectively \textit{almost star-free}) if it has finite intersection with any conjugate of a join subgroup (respectively a star subgroup).
\end{defn}

\begin{rem}
\label{rem:star-vs-join}
    In the case that the graph has no isolated vertices, every star subgroup is a join subgroup. Consequently, we also have that (almost) join-free implies (almost) star-free.
\end{rem}

\subsection{Associated graphs and complexes} \label{sec:graph-complexes}

Throughout this paper we consider several different Cayley graphs associated to $\gp$ and their corresponding metrics. For more details on these the reader can refer to \cite{genevois2017} or \cite{valiunas2021graphproducts}. We recall here the basic definitions and facts we will need.

When necessary, we will assume that the vertex groups of $\gp$ are finitely generated, and fix finite generating sets $\{S_v\}_{v\in V(\Gamma)}$ for those vertex groups. In this case the graph product $\gp$ is also finitely generated, and comes with a natural finite generating set 

\[ S := \bigcup\limits_{v\in V(\Gamma)} S_v. \label{standars generating set}  \]

When we talk about the \emph{standard Cayley graph} of (a finitely generated) $\gp$, we refer to the Cayley graph $\Cay_S(\gp)$, which is the Cayley graph with respect to this finite generating set $S$. For $g\in\gp$ we denote distance from the trivial element to $g$ in $\Cay_S(\gp)$ by $|g|_S$. By a \emph{word in $S$} we mean the trivial word, or a word written in the letters of $S$ or their inverses.

\begin{rem}
    Although we need to assume that $\gp$ is finitely generated for \Cref{thm-intro:recognizing-stable}, the majority of intermediate results in this paper do not require this assumption. In particular, it is not needed in the proof that (\ref{thm-part-intro:recognizing-stable-ploxo}) $\implies$ (\ref{thm-part-intro:recognizing-stable-qi}), which is what the majority of this paper works towards.
\end{rem}

\subsubsection{Prism Cayley graph and complex}

\label{def:prism complex}

Most of the work in this document uses what we call the \emph{prism Cayley graph} of $\gp$, following Genevois' study of graph products and other groups in \cite{genevois2017}. This is the Cayley graph using the generating set $S_p$ (the \emph{prism generating set} for $\gp$) defined as follows:
\[   S_p := \bigcup\limits_{v\in V(\Gamma)} G_v\setminus \{\id_{G_v}\}. \label{prism generating set}\]
The prism Cayley graph is denoted $\Cay_p(\gp)$, geodesic words for the prism generating set are referred to as \emph{prism geodesics} or \emph{$S_p$--geodesics}, and $|\cdot|_p$ is the prism length for a group element. Genevois proved that the prism Cayley graph of a graph product is always a quasi-median graph \cite[Proposition 8.2]{genevois2017}.

To this prism Cayley graph is associated a cell complex: we say that the \emph{prism complex} $P(\gp)$ of $\gp$ is the 2--complex whose 1--skeleton is the prism Cayley graph (the 1--skeleton is labeled by the elements of the prism generating set) and whose 2--cells correspond to the following relations:

\begin{align} 
    &[g,h]=\id \text{ when } g\in G_v, h\in G_u \text{ with } \{u,v\}\in E(\Gamma)  \label{relation:Square} \tag{$\openbox$}\\
     &  g\cdot g' = (gg') \text{ when } g,g'\in G_v\setminus \{\id_{G_v}\}, g'\neq {g^{-1}} \label{relation:Triangle} \tag{$\triangle$}
\end{align}

This complex is therefore made out of squares and triangles as illustrated in \Cref{fig:PrismComplex} (which motivates its name). Edges of a triangle are labeled by elements in the same vertex groups, as are elements on opposite edges of a square. In particular, labels on opposite edges of a square are labeled by the same element, but it is possible for labels on edges of a triangle to all be distinct.

\begin{rem}
    In this paper we refer to such a cell complex as the prism complex of $\gp$ for convenience, however they are in fact a subclass of what Genevois calls prism complexes of quasi-median graphs \cite{genevois2017}. In fact, what we call the prism complex is the 2--skeleton of Genevois' quasi-median complex, which is obtained by filling every clique with a simplex, and filling every 1--skeleton of an $n$--cube with an $n$--cube. As we do not need to consider any higher dimensional cells, we omit them here.
\end{rem}

\begin{figure}[!h]
    \centering
    \includegraphics[width=0.5\linewidth]{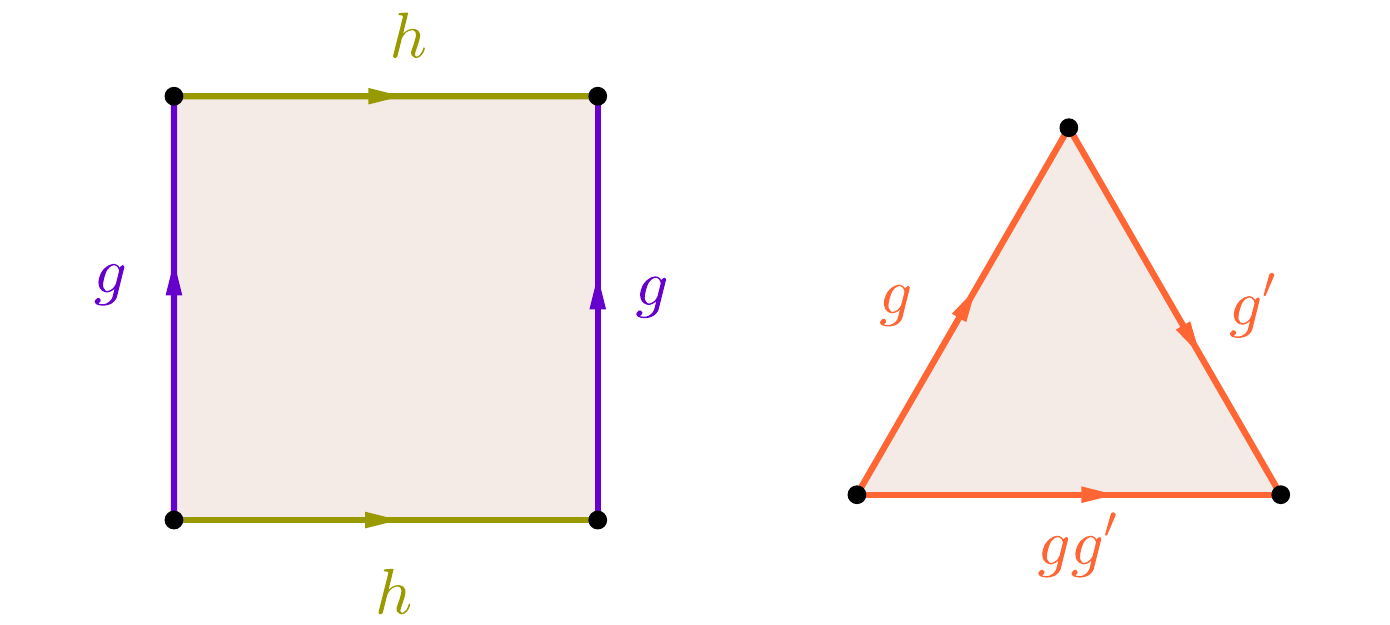}
    \caption{2--cells of the prism complex}
    \label{fig:PrismComplex}
\end{figure}

\subsubsection{Hyperplanes and carriers}\label{subsubsec:hyperplanes}
In the prism complex, a \emph{hyperplane} represents an equivalence class of edges: two edges are equivalent if they belong to the same triangle or are opposite edges of the same square. All the edges in a given hyperplane are labeled by generators from the same vertex group; this defines the vertex (and vertex group) \emph{associated to} that hyperplane.  Each hyperplane corresponds to a dual \emph{geometric hyperplane}: the union of the \emph{midcubes} intersecting the edges of that hyperplane in $P(\gp)$, where midcubes for the three kinds of cells are as illustrated in Figure \ref{fig:MidCubes}.

\begin{figure}[!ht]
    \centering
    \includegraphics[width=0.5\linewidth]{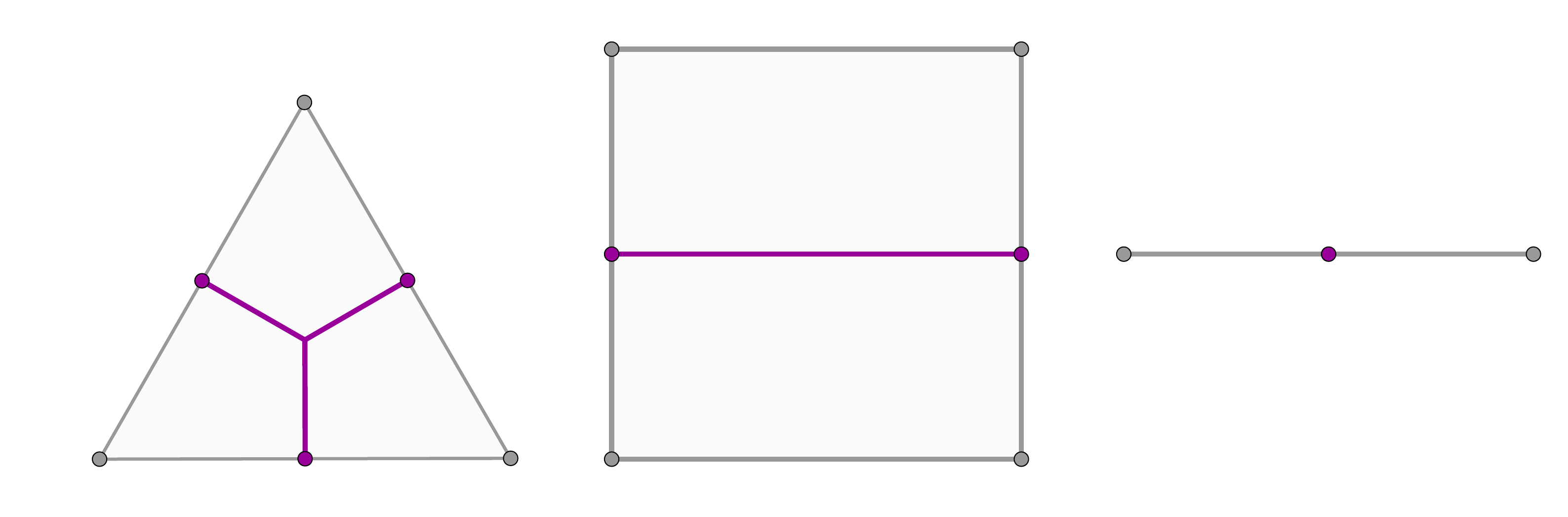}
    \caption{Midcubes for 1-- and 2--cells of $P(\gp)$}
    \label{fig:MidCubes}
\end{figure}

A hyperplane may also be identified with its \emph{carrier}: the union of closed cells that intersect the geometric hyperplane. From definitions one can verify that the 0--skeleton of a carrier is precisely some coset of $G_{\Star(v)}$, where $v$ is the vertex associated to the hyperplane \cite{genevois2017,GenevoisMartin}.  We also make use of its 1--skeleton, the full subgraph induced in $P(\gp)$ by these vertices (following \cite{valiunas2021graphproducts}). Note that the edges of the 1--skeleton that do not belong to the hyperplane itself are labeled by generators from $G_{\Link(v)}$.

Two distinct hyperplanes $\mathfrak{h}$ and $\mathfrak{h}'$ are said to be \emph{transverse} if there exist a pair of equivalent edges in $\mathfrak{h}$ and a pair of equivalent edges in $\mathfrak{h}'$ which together form a square in the prism complex; equivalently, their corresponding geometric hyperplanes intersect. Note that this can only occur when the vertices corresponding to these two hyperplanes span an edge in $\G$, so the vertex groups commute. Distinct hyperplanes are \emph{tangent} if they are not transverse, but their carriers intersect.

\subsubsection{Contact graph}\label{subsubsec:ContactGraph}

The \emph{contact graph} $\CX$ of the graph product $\gp$ is the graph whose vertices correspond to the hyperplanes of $P(\gp)$, and whose edges connect pairs of vertices whose associated hyperplanes are either transverse or tangent; equivalently, there is an edge between two hyperplanes when their carriers intersect, an example is drawn in \Cref{fig:ContqctCrossingGrqphs}. The natural action of the group $\gp$ on $P(\gp)$ extends to an isometric action on $\CX$.

\begin{figure}[h]
    \centering
    \includegraphics[scale=0.2]{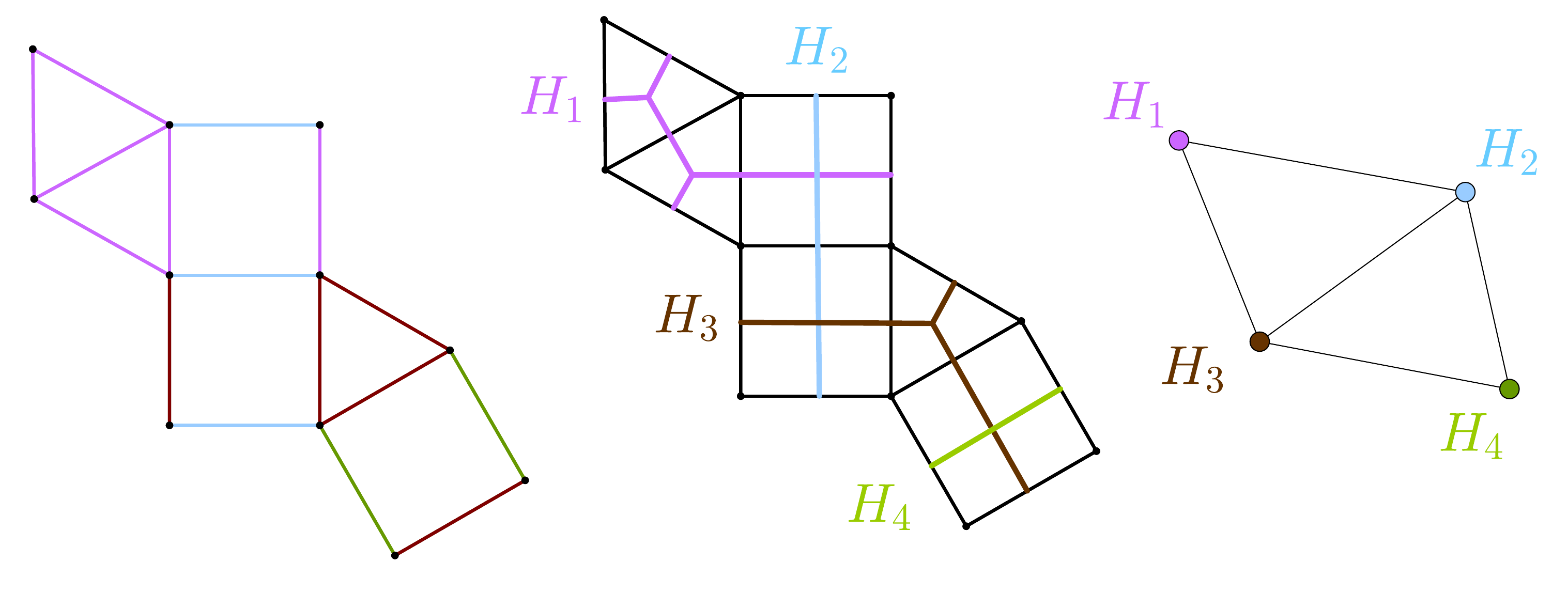}
    \caption{Hyperplanes, geometric hyperplanes, and contact graph. $H_1$ and $H_2$, $H_2$ and $H_3$, and $H_3$ and $H_4$ are transverse, while $H_1$ and $H_3$, and $H_2$ and $H_4$ are tangent.} 
    \label{fig:ContqctCrossingGrqphs}
\end{figure}

The definition above is parallel to the construction of the contact graph of a CAT(0) cube complex, introduced by Hagen in \cite{Hagen}. Several results from the cube complex case also hold for graph products, in particular the following is known.

\begin{thm}\emph{\cite[Corollary C]{valiunas2021graphproducts}}\label{thm:valiunas}
Let $\Gamma$ be a bounded degree simplicial graph, and let $\gp$ be a graph product of non-trivial groups. Then $\CX$ is a quasi-tree (and therefore a hyperbolic space), and the action of $\gp$ on $\CX$ is acylindrical.
\end{thm}

\subsubsection{Star Cayley graph} \label{star generating set}

The last generating set we consider for $\gp$ comes from star subgroups:
\[ S_\star := \bigcup\limits_{v\in V(\Gamma)} G_{\Star(v)}\setminus\{\id_{G_{\Star(v)}}\}.   \]
We call its associated Cayley graph $\Cay_\star(\gp)$ the \emph{star Cayley graph}, and apply analogous terminology and notations as for the previous Cayley graphs: geodesic words for this generating set are \emph{star geodesics}, and $|\cdot|_\star$ denotes the star length of an element.

The following was observed by Valiunas in \cite[Proof of Corollary D]{valiunas2021graphproducts}, and follows from \cite[Section 8.1]{genevois2017} and \cite[Theorem 2.10]{GenevoisMartin}. We include a short proof here for completeness.

\begin{lem}\emph{\cite{genevois2017,GenevoisMartin}}
\label{lem:Contact=star}
 For a graph product $\gp$, the contact graph of the prism complex is $\gp$-equivariantly quasi-isometric to the star Cayley graph.
\end{lem}

\begin{proof}
    Let $d$ be the metric on $\CX$. Identify hyperplanes in $P(\gp)$ with the 0--skeletons of their carriers, which, as mentioned in \Cref{subsubsec:hyperplanes}, are of the form $gG_{\Star(v)}$ for some $g\in\gp$ and $v\in V(\G)$. Recall that two hyperplanes are adjacent in $\CX$ if and only if their carriers intersect, and note that any such intersection implies the intersection of their 0--skeletons. Fix $v\in V(\G)$, and define $f:\Cay_{\star}(\gp)\to\CX$ by $f(g)=gG_{\Star(v)}$. This map is clearly $\gp$-equivariant, so we just have to verify that it is a quasi-isometry.

    Suppose $g\in \gp$ is such that $|g|_{\star}=1$, so $g\in G_{\Star(w)}$ for some $w\in V(\G)$. Then $G_{\Star(v)}\cap G_{\Star(w)}$ and $G_{\Star(w)}\cap gG_{\Star(v)}$ are both nonempty (containing $\id$ and $g$ respectively), so $d(G_{\Star(v)},gG_{\Star(v)})\leqslant 2$. For general $g\in\gp$ the triangle inequality therefore tells us that $d(G_{\Star(v)},gG_{\Star(v)})\leqslant 2|g|_{\star}$.

    For the opposite inequality, let $g\in G\backslash\{\id\}$, and let $k = d(G_{\Star(v)},gG_{\Star(v)})$. This gives us a sequence of $k+1$ hyperplanes $G_{\Star(v)}=H_0,\ldots,H_{k}=gG_{\Star(v)}$ such that for every $i$ the carriers of $H_i$ and $H_{i+1}$ intersect. By following paths along the 1--skeletons of these carriers, we can therefore write $g$ as $h_0\cdots h_k$, where each $h_i$ belongs to the star subgroup corresponding to $H_i$. Therefore $|g|_{\star}\leqslant k+1\leqslant d(G_{\Star(v)},gG_{\Star(v)})+1$.

    Finally, let $g\in\gp$ and $w\in V(\G)$. Then $gG_{\Star(w)}\cap gG_{\Star(v)}\neq\emptyset$, so $gG_{\Star(w)}$ is either adjacent to, or equal to, $f(g)$.
\end{proof}

\begin{rem}
    This proof also shows that the contact graph is connected.
\end{rem}

%%%%%%%%%%%%%%%%%%%%%%%%%%%%%%%%%%%%
\section{Disk diagrams}
\label{sec:DD}
%%%%%%%%%%%%%%%%%%%%%%%%%%%%%%%%%%%%

Many of the proofs in this paper are visualized using \emph{disk diagrams}, a version of the \emph{dual van Kampen diagrams} used for graph products by Valiunas \cite[Section 6.2]{valiunas2021graphproducts}, and also by Genevois \cite{Genevois_VanKampen}.  For finitely presented groups this notion is classical (cf. the \emph{pictures} of \cite{MR787801}), but our particular use of disk diagrams for graph products closely mirrors their more recent use for right-angled Artin groups (\cite{CrispWiest}, \cite{Kim}, \cite{KimKoberda2}).  In \Cref{sec:constructing_DD} we review the construction of such diagrams, and in \Cref{sec:DD_lemmas} we note some basic lemmas that will be important to our applications. In \Cref{sec:combing}, \Cref{sec:star-free}, and \Cref{sec:proof} we prove some more technical disk diagram results that will be key to the proof of the main theorem.

\subsection{Constructing disk diagrams}
\label{sec:constructing_DD}

By the nature of the relators for a presentation of $\gp$ using the prism generating set $S_p$, a van Kampen diagram $\vkd$ for an identity word $s=s_1\dots s_k$ will be a simply-connected complex made of squares, triangles, and edges embedded in $\mathbb{R}^2$, with edges labeled by the letters $s_i \in S_p$ around the boundary (see, for example, \cite{LS}).  Our convention will be to read boundary words clockwise in van Kampen diagrams, their duals, and disk diagrams alike.  Observe that (geometric) hyperplanes and their carriers can be defined for $\vkd$ in the same way as they are for the prism complex, using a cellular map from $\vkd$ to $P(\gp)$ that respects edge labels.  In particular, each geometric hyperplane and hyperplane carrier in a van Kampen diagram is a connected component of the preimage of a geometric hyperplane or hyperplane carrier in the prism complex, and two geometric hyperplanes can only intersect in a van Kampen diagram if their corresponding hyperplanes in the prism complex are transverse.

The dual of such a van Kampen diagram can be constructed as follows (and as illustrated in \Cref{fig:fronVKtoDVK}): enlarge $\vkd$ to the embedded complex $\vkd'$ by adding one more vertex $v_\infty$ to $\R^2 \setminus \vkd$, then adding edges between this vertex and each vertex encountered along the boundary of $\vkd$ (traversed cyclically), and finally adding 2--cells for each triangle bounded by a $\vkd$-boundary edge plus the two newly-added edges adjacent to it.  The \emph{dual van Kampen diagram} $\dvkd$ is the dual of $\vkd'$ with the face corresponding to $v_\infty$ deleted: it is a tessellated disk whose 1--skeleton consists of $\partial\dvkd$ plus an interior subgraph, which may be disconnected. 

\begin{rem} \label{rem:dual=GeomHyp}
    Note that if one takes the van Kampen diagram together with its geometric hyperplanes, then $\Delta$ can be embedded in $\mathbb{R}^2$ such that it contains $\vkd$, as in \Cref{fig:fronVKtoDVK}, and such that the intersection of the 1--skeleton of $\Delta$ together with $\vkd$ is exactly the union of the geometric hyperplanes of $\vkd$.
\end{rem}

\begin{figure}[!h]
    \centering
   \includegraphics[width=1\linewidth]{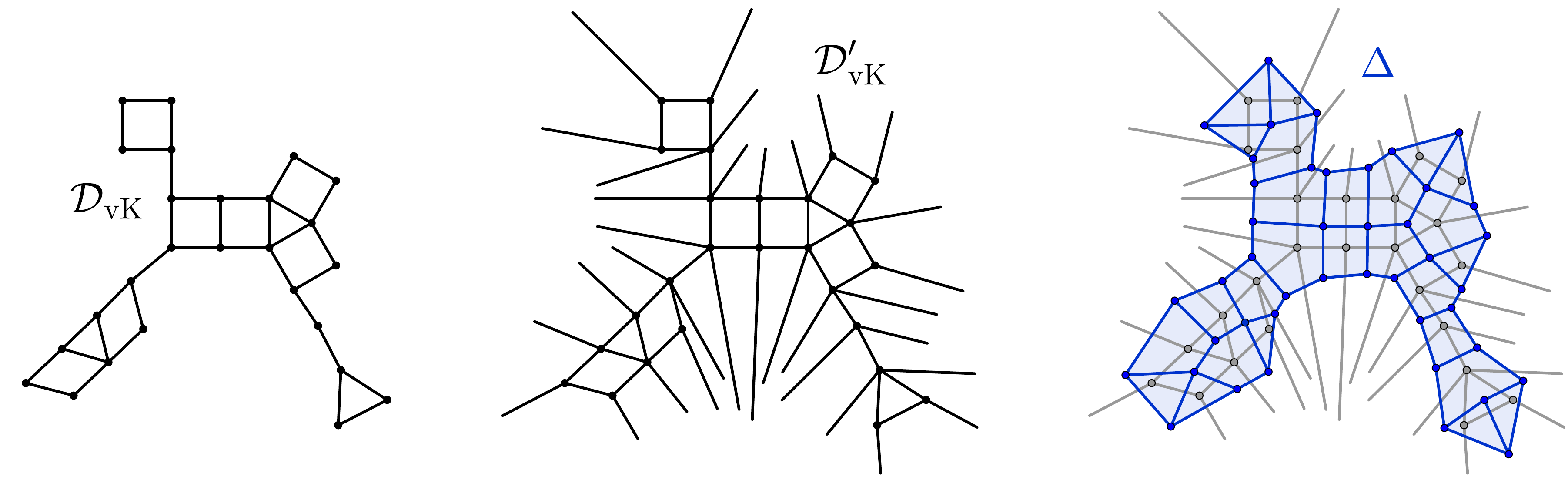}
    \caption{From van Kampen diagram to dual van Kampen diagram}
    \label{fig:fronVKtoDVK}
\end{figure}

The boundary of a van Kampen diagram $\vkd$ is generally not an embedded circle, but $\partial\dvkd$ is a topological circle (as it effectively comes from a neighborhood of $\vkd$), and furthermore we can identify $\dvkd$ with the topological disk $\dd_0$.  We will define a disk diagram for $w$ to be a certain labeling of the topological disk $\dd_0$ along with an ``almost cellular'' 
map $\phi:\dd_0 \to \vkd$ with certain properties (Proposition \ref{prop:disk_diagram_existence} below), with dual van Kampen diagrams serving as an intermediate step in the proof of their existence.

\begin{prop}[Disk diagram existence]\label{prop:disk_diagram_existence} Let $s$ be an identity word in the prism generating set for a graph product $\gp$, and let $\vkd$ be a van Kampen diagram for $s$. There exists a map $\phi$ from the disk $\dd_0$ to $\vkd$ such that:
\begin{itemize}
    \item[(a)] the image of $\partial \dd_0$ is the circuit along the boundary of $\vkd$ following edges labeled in order by the letters of $s$;
    \item[(b)] $\phi$ is a local homeomorphism on the pre-image of the interior of any 2--cell in $\vkd$;
    \item[(c)] the pre-image of any geometric hyperplane in $\vkd$ is an embedded graph in $\dd_0$.
\end{itemize} 
\end{prop}

\begin{proof}
    Letting $\vkd$ be as in the statement, we start with the dual van Kampen diagram construction $\dvkd$ described above, recalling that $\dvkd$ can be seen as a neighborhood of $\vkd$, and in particular observing that $\dvkd$ deformation retracts to $\vkd$. We also have that there is a natural homeomorphism between $\dd_0$ and $\dvkd$.
    
    As noted above (see \Cref{rem:dual=GeomHyp}), we can arrange the composition $\dd_0 \to \dvkd \to \vkd$ such that the geometric hyperplanes in $\vkd$ pull back to subgraphs of the inner 1--skeleton of $\Delta$, and therefore pull back to embedded graphs in $\dd_0$.  We let $\phi:\dd_0 \to \vkd$ come from the composition, which concludes the proof.
\end{proof}

This enables the following definition. 

\begin{defn}  
    Let $s=s_1\cdots s_k$ be an identity word in $S_p$. A \emph{disk diagram with boundary word} $s$ is a disk $\dd$ with a marking by a map $\phi$ as in \Cref{prop:disk_diagram_existence} such that
    \begin{itemize}
        \item $\partial \dd$ is segmented by the pre-images of the boundary edges of $\vkd$, and each segment is marked by the corresponding letter $s_i$ of $s$,
        \item the interior of $\dd$ is marked by the pre-images of the geometric hyperplanes.
    \end{itemize}
The connected components of such pre-images of hyperplanes are called \emph{dual graphs}, their dual edges in the boundary and their corresponding generator labels are called \emph{roots} of the dual graph, and the dual graph is said to be \emph{rooted} in its dual edges of $s$.
\end{defn}

\begin{rem}
    Note that properties of the geometric hyperplanes also apply to the dual graphs. In particular, each dual graph is associated to a given vertex of $\Gamma$, and all its roots belong to the corresponding vertex group. Moreover, distinct dual graphs may only intersect if their corresponding vertices have an edge between them in $\Gamma$.
\end{rem}

\begin{rem}
    The construction of dual van Kampen diagrams serves to prove the existence of disk diagrams, but the point the reader should keep in mind is the correspondence between van Kampen diagrams and disk diagrams. Most of the proofs in the rest of paper are based on disk diagrams. The pictures to keep in mind are illustrated in \Cref{fig:vKtodisk}. Here the colors can be thought of as corresponding to the vertex groups associated to the hyperplanes.

    \begin{figure}[!h]
    \centering
           \includegraphics[scale=0.4]{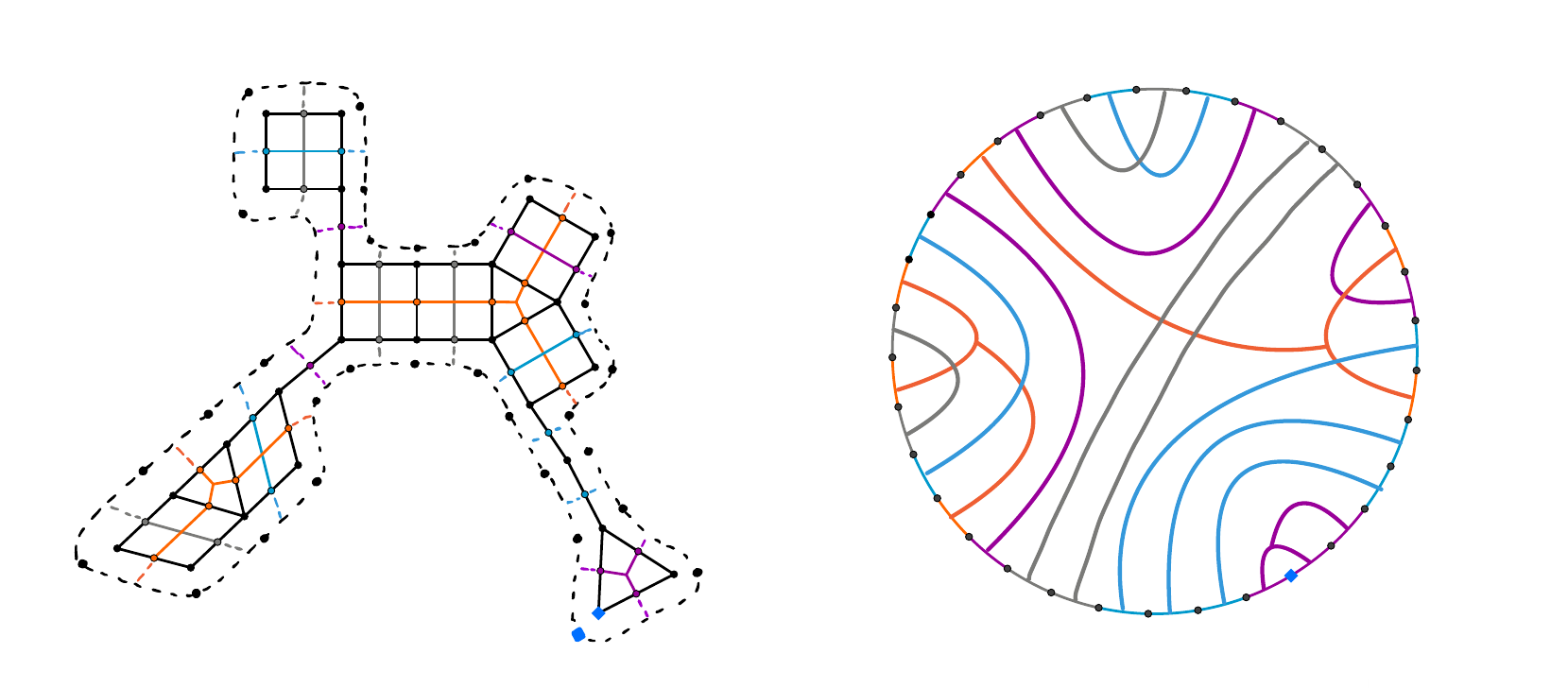}
    \caption{From (dual) van Kampen diagram to disk diagram}
    \label{fig:vKtodisk}
    \end{figure}
\end{rem}

\begin{conv*}
    As already mentioned, we will work with disk diagrams oriented clockwise. Most of the time we will use disk diagrams to study identity words of the form $s=g\overline{w}$, where $w=w_1\cdots w_m$ is a prism geodesic representative for $g$, and $\overline{w}$ is its inverse; then $g$ will be read clockwise and $w$ counter-clockwise as illustrated in \Cref{fig:convention}.
\end{conv*} 

\begin{figure}[!h]
    \centering
           \includegraphics[scale=0.4]{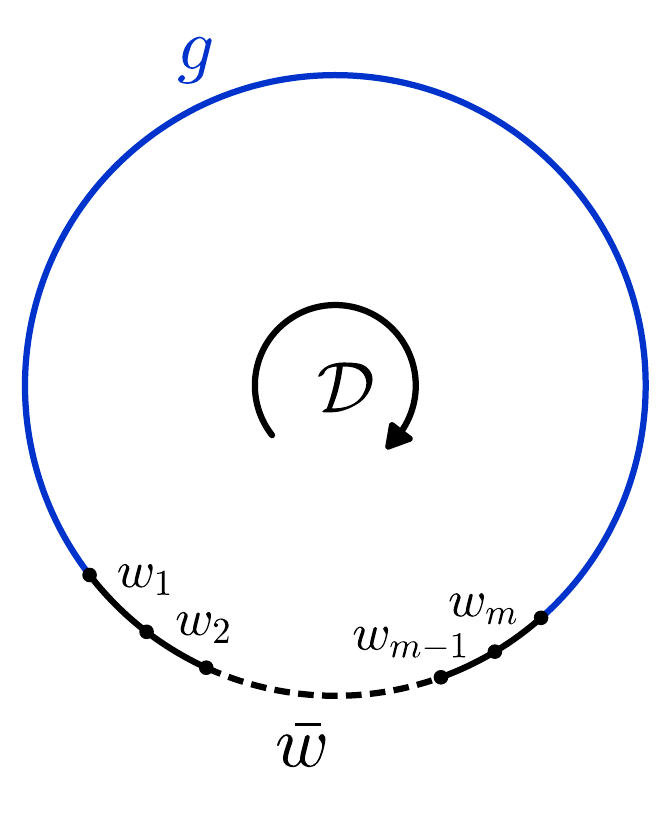}
    \caption{Convention for disk diagrams}
    \label{fig:convention}
\end{figure}

We will be interested especially in paths in $\dd$ that map to edge paths in $\vkd$.  Already we may consider $\partial\dd$ as having inherited a cellular decomposition into vertices and edges, corresponding to its image circuit in the van Kampen diagram. Between vertices on $\partial\dd$ that are adjacent to edges that are roots of the same dual graph $\alpha$, a \emph{word along the carrier of the dual graph} means the product of a sequence of generators labeling an edge path in $\vkd$ connecting the corresponding vertices of $\partial\vkd$, which follows the 1--skeleton of the carrier of the hyperplane $\phi(\alpha)$. See examples in \Cref{fig:word_carrier}. In particular, if $\alpha$ corresponds to $v\in V(\gp),$ then words along its carrier spell elements of $G_{\Star(v)}$. 
\begin{figure}[!h]
           \includegraphics[scale=0.56]{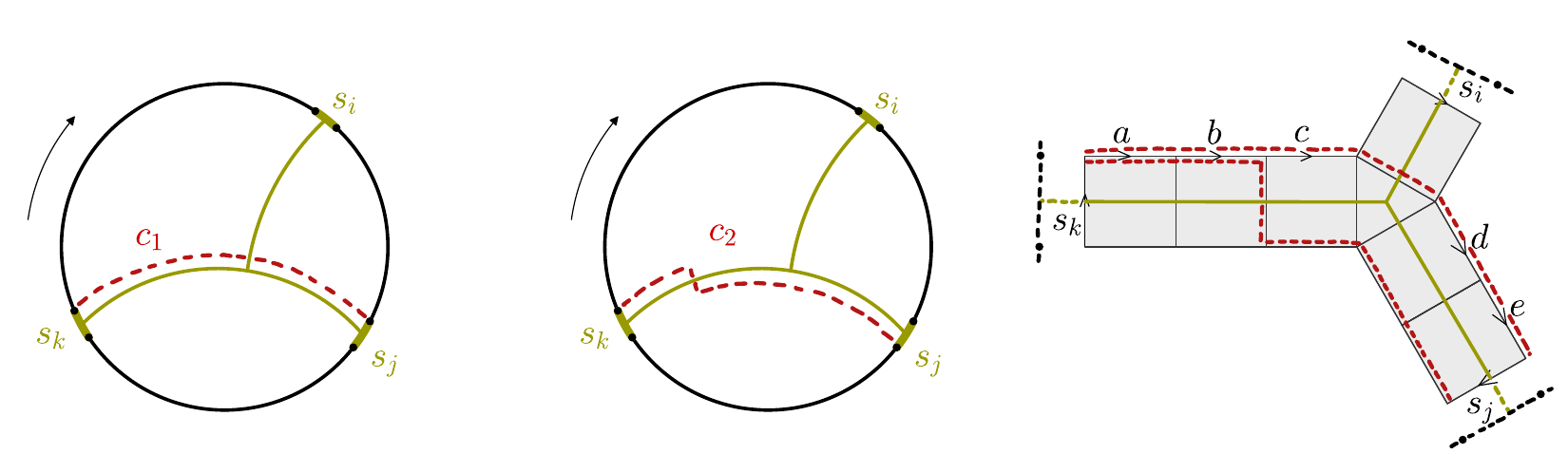}
    \caption{Words along a carrier: $c_1=abcs_ide$, $c_2=ab\bar{s}_kcde$}
    \label{fig:word_carrier}
\end{figure}

\subsection{Disk diagram lemmas}

\label{sec:DD_lemmas}

Here we collect some fundamental properties of disk diagrams that will be used in this paper. First, we note that disk diagrams concatenate in the same way van Kampen diagrams do.

\begin{lem}[Concatenating disk diagrams]\label{lem:disk_concatenation}
    Let $s_1$ and $s_2$ be identity words in the prism generating set for $\gp$, with respective disk diagrams $\mathcal{D}_1$ and $\mathcal{D}_2$ given by Proposition \ref{prop:disk_diagram_existence}. Suppose $b$ is a subword of $s_1$ so that $s_1$ is the product $a_1bc_1$ of three subwords (with one of $a_1$ and $c_1$ possibly empty), and $\bar b$ a subword of $s_2$ with $s_2=a_2\bar bc_2$.  Then identifying $\mathcal{D}_1$ and $\mathcal{D}_2$ along the subsets of their boundaries labeled by the letters of $b$ provides a disk diagram for $a_1 c_2 a_2 c_1$.
\end{lem}   

\begin{proof} The disk diagram for Lemma \ref{lem:disk_concatenation} corresponds to the map $\phi$ defined as follows. The marking of $\dd_1$ comes from some map $\phi_1:\dd_1 \to \vkd^{(1)}$, where $\vkd^{(1)}$ is a van Kampen diagram for $s_1$. Similarly, the marking of $\dd_2$ comes from some map $\phi_2:\dd_2 \to \vkd^{(2)}$, where $\vkd^{(2)}$ is a van Kampen diagram for $s_2$. These van Kampen diagrams can be identified along the subsets of their boundaries labeled by $b$ to obtain a van Kampen diagram $\vkd$ for $a_1 c_2 a_2 c_1$. Then $\phi_1\cup \phi_2:\dd_1\cup \dd_2 \to \vkd$ induces a well-defined map $\phi:\dd \to \vkd$ where $\dd=(\dd_1\cup \dd_2)/\sim$ is the 2--to--1 identification along $b$ of points in $\partial \dd_1$ and $\partial \dd_2$, which have the same image in $\vkd$ under $\phi_1$ and $\phi_2$ respectively.
\end{proof}

We highlight the next observation as we will use it repeatedly.

\begin{fact}
\label{fact:concatenate-dual-graph}
    If $\dd$ is the concatenation of $\dd_1$ and $\dd_2$ then its dual graphs are concatenations of dual graphs from $\dd_1$ and $\dd_2$.
\end{fact}

Dual graphs in disk diagrams can be thought of as recording ``cancellation patterns'' for $S_p$--letters of identity words, as illustrated in the following lemma.

\begin{lem}\label{lem:single_hyperplane}
    Let $\alpha$ be a dual graph in a disk diagram $\dd$, and let $a_1,\dots,a_n$ be all the roots of $\alpha$, in order around the boundary of $\dd$. Then every $a_i$ belongs to the same vertex group, and $a_1\cdots a_n$ is an identity word. In particular, any dual graph has at least two roots.
\end{lem}

\begin{proof}
    The first claim restates the fact that generators labeling edges in the same hyperplane must belong to the same vertex group, say $G_v$. To show that $a_1\cdots a_n$ is the identity, consider the (potentially empty) boundary subwords $c_i$ between $a_i$ and $a_{i+1}$, along with $c_n$ between $a_n$ and $a_1$, so $a_1c_1a_2\ldots c_{n-1}a_nc_n=\id$. We can then consider a word $c_i'$ along the carrier of $\alpha$ from the inner end of $a_i$ to the inner end of $a_{i+1}$, so that $c_i'\overline{c_i}$ forms a closed circuit, and thus $c_i$ and $c_i'$ are group-equivalent. Therefore $a_1c_1'a_2\ldots c_{n-1}'a_nc_n'=\id$. 
    
    Note that as $\{a_1,\dots,a_n\}$ are all of the roots of $\alpha$, the subword $c_i$ does not contain any roots of $\alpha$. We can therefore choose the path $c_i'$ in the van Kampen diagram corresponding to $\dd$ such that it does not cross the image of $\alpha$, which means that every edge it traverses is labeled by an element of $G_{\Link(v)}$. As $G_v$ and $G_{\Link(v)}$ form a direct product subgroup of $\gp$, we get that $a_1\cdots a_nc_1'\cdots c_n' = \id$, which implies that $a_1\cdots a_n = \id$ due to the direct product structure.  Moreover, because $\id$ is not a generator, $n\geq 2$.
\end{proof}

Our next lemma, about geodesic subwords of an identity word, recovers as a corollary that two vertices in $\Cay_p(\gp)$ have only finitely many geodesics between them (whereas $\Cay_p(\gp)$ is locally infinite).  This fact is also a direct consequence of the normal form for graph products \cite{greennormalform}.

\begin{lem}[Single-rooted graphs]\label{lem:rooting}Suppose a disk diagram has along its boundary an $S_p$--geodesic subword $w$.  Then each dual graph roots in at most one edge of $w$.
\end{lem}

\begin{proof}
    We prove the contrapositive. Let $\dd$ be a disk diagram whose boundary word has a subword $w$. Suppose that $\alpha$ is a dual graph with roots in $w$ labeled by generators $l$ and $r$, appearing in that order in $w$, and let $v\in \G$ be the vertex corresponding to $\alpha$. Let $c$ be the (possibly empty) subword of $w$ between $l$ and $r$ on $\partial\dd$. We can assume without loss of generality that $c$ does not contain any roots of $\alpha$. We can then consider a word $c'$ along the carrier of $\alpha$ from the inner end of $l$ to the inner end of $r$, so that $c'\overline{c}$ forms a closed circuit, and thus $c$ and $c'$ are group-equivalent. As in the proof of \Cref{lem:single_hyperplane}, we can choose the path $c'$ in the van Kampen diagram so that it does not cross the image of $\alpha$, which means that every edge it traverses is labeled by an element of $G_{\Link(v)}$. As $l,r\in G_v$, we see that $c'$ commutes with both of these. Thus, as group elements, $lcr = lc'r = lrc' = kc' = kc,$ where $k = lr$ is also in $S_p$. The word $kc$ is strictly shorter than $lcr$, implying that $w$ is not geodesic.
\end{proof}

\begin{cor}{\emph{\cite[Theorem 3.9]{greennormalform}}}\label{lem:finite_geodesics}
    The $S_p$--geodesic word for a given element in $\gp$ is unique up to ordering of its letters, so there are only finitely many prism geodesics between two vertices in $\Cay_p(\gp)$.
\end{cor}

\begin{proof}
    Suppose $u$ and $v$ are prism geodesics between the same pair of vertices, so that $u\overline{v}$ labels the boundary word of a disk diagram.  By \Cref{lem:single_hyperplane} and \Cref{lem:rooting}, each dual graph of this diagram has exactly two roots, one in $u$ and the other in $\overline{v}$. \Cref{lem:single_hyperplane} ensures that the letters associated to those to roots are inverse to each other. Therefore $u$ and $v$ consist of the same letters and only differ in their ordering.
\end{proof}

From these lemmas one may derive Genevois' characterization of prism-geodesic words; we give a proof for completeness of exposition.
\begin{lem}\emph{\cite[Theorem 3.2]{Genevois_VanKampen} }\label{lem:charac_geod} Let $w=w_1\cdots w_n$ be a non-trivial prism word of a graph product $\gp$. This word is geodesic if and only if for any $i<j$, if $w_i$ and $w_j$ belong to the same vertex group $G_v$, then there exists $i<k<j$ whose associated vertex in $\Gamma$ does not belong to $\Star{(v)}$.
\end{lem}

\begin{proof}
    Take a prism word $w=w_1\cdots w_n$ admitting letters $w_i$ and $w_j$ which belong to the same vertex group $G_v$, and such that all the intermediate letters $w_k$ with  $k\in\{i+1,\ldots,j-1\}$ belong to vertex groups in $G_{\Star(v)}$. Hence, by construction $w_i$ commutes with $w_{i+1}\cdots w_{j-1}$. Treating $w_iw_j$ as one generator, the prism word $w_1\cdots w_{i-1}w_{i+1}\cdots w_{j-1}(w_iw_j)w_{j+1}\cdots w_n$ is a prism word of length at most $n-1$ representing the same group element as $w$, which is therefore not a prism-geodesic word.

    Now suppose that $w$ satisfies the latter condition of the lemma, and take $u$ to be a prism-geodesic word representing the same group element. Fix $\mathcal{D}$ a disk diagram for $w\overline{u}$. By \Cref{lem:rooting} we know that any graph dual to a letter in $u$ has only one root in $u$. If the same is also true for $w$, then every dual graph in $\mathcal{D}$ has exactly one root in $w$ and one in $u$, meaning that $w$ and $u$ have the same prism length, which would tell us that $w$ is geodesic.
    
    Assume for contradiction that there are two distinct letters $a$ and $b$ of $w$ dual to the same dual graph $\alpha$ (in particular they belong to the same vertex group $G_v$);
    we can choose $a$ and $b$ such that the distance between $a$ and $b$ in $w$ is the shortest distance relative to pairs of letters coming from dual graphs with at least two roots in $w$.
    Since $a$ and $b$ belong to the same vertex group, our assumption on $w$ implies that there must be a letter $c$ between $a$ and $b$ such that the vertex of $\Gamma$ associated to $c$ is not in $\Star(v)$.
    \Cref{lem:single_hyperplane} ensures that the graph dual to $c$ has a second root in the boundary of the diagram. By the minimality of $a$ and $b$, this second root cannot be between $a$ and $b$. This implies that the graph dual to $c$ crosses $\alpha$, meaning $c$ belongs to a vertex group whose associated vertex is adjacent to $v$, contradicting that the associated vertex is not in $\Star(v)$.
\end{proof}

\subsection{Combing}
\label{sec:combing}

Our goal in this section is to construct disk diagrams that have certain favorable properties, with respect to some geodesic subword $w$ appearing backwards along their boundary. (Recall that we read $w$ along the opposite orientation to fit our typical use of boundary words $g\overline{w}$ starting with some non-geodesic part $g$.) The ability to construct such disk diagrams will be critical to the proof of Theorem \ref{thm:join-busting}.

Let $g=g_1\cdots g_n$ be a prism word and $w=w_1\cdots w_m$ a prism-geodesic representative for $g$ (so both $g_j$ and $w_i$ are in $S_p$). Given a disk diagram $\mathcal{D}$ for $g\overline{w}$, we define the \emph{beginning function} (with respect to $w$) $b^\mathcal{D}: \{1,\ldots,m\}\to\{1,\ldots,n \}$ and the \emph{ending function} (with respect to $w$) $e^\mathcal{D}: \{1,\ldots,m\}\to\{1,\ldots,n \}$ which associate to each index $i\in \{1,\ldots,m\}$ the index $j\in \{1,\ldots,n \}$ of the first, and respectively last, letters of $g$ which are dual to the same dual graph as $w_i$ in $\mathcal{D}$.  In order to use beginning and ending functions flexibly, i.e. for any geodesic subword of a disk diagram boundary, we observe the following:

\begin{fact} \label{fact:Cyclic boundary}
    A disk diagram $\dd$ with boundary word $s$ is also a disk diagram for the boundary word $s'$ where $s'$ is any cyclic permutation of $s$.
\end{fact}

\begin{defn}
    Let $\overline{w}$ be a prism-geodesic subword of the boundary word for a disk diagram $\mathcal{D}$, so that after cyclic permutation this boundary word is $g\overline{w}$. Let $b^\dd$ and $e^\dd$ be the beginning and ending functions of $\dd$ with respect to $w$. We say $\mathcal{D}$ is \emph{$w$--left-combed} if, for any $1 \leq i < j \leq m$, $b^\mathcal{D}(i)<b^\mathcal{D}(j)$; likewise $\mathcal{D}$ is \emph{$w$--right-combed} if for any $i<j$, $e^\mathcal{D}(i)<e^\mathcal{D}(j)$ as illustrated in \Cref{fig:w-combed}.
\end{defn}

  \begin{figure}[!h] 
\begin{subfigure}[c]{.32\textwidth}
  \centering
  \includegraphics[scale=0.4]{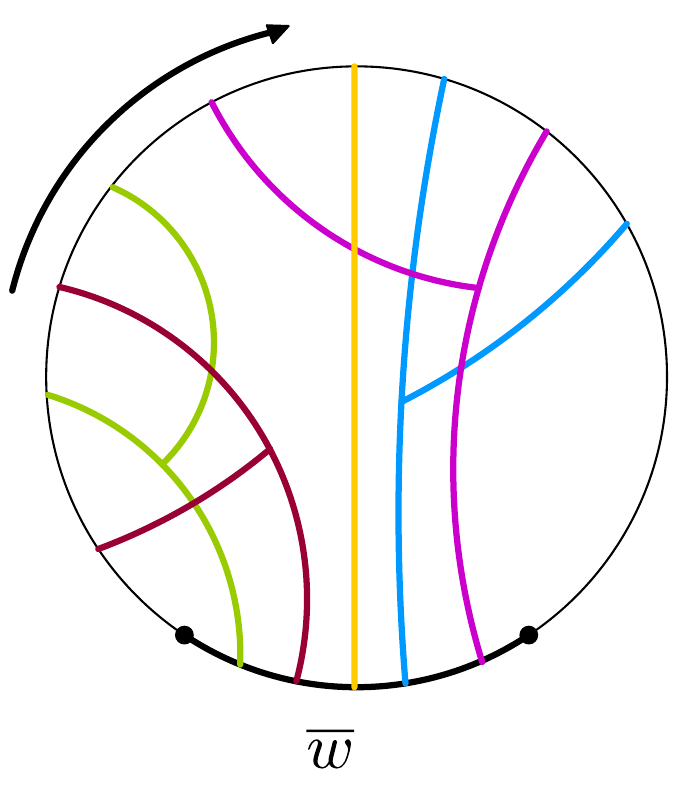}
  \caption{Not combed }
  \label{fig:NotCombed}
\end{subfigure}%
\hfill
\begin{subfigure}[c]{.32\textwidth}
  \centering
  \includegraphics[scale=0.4]{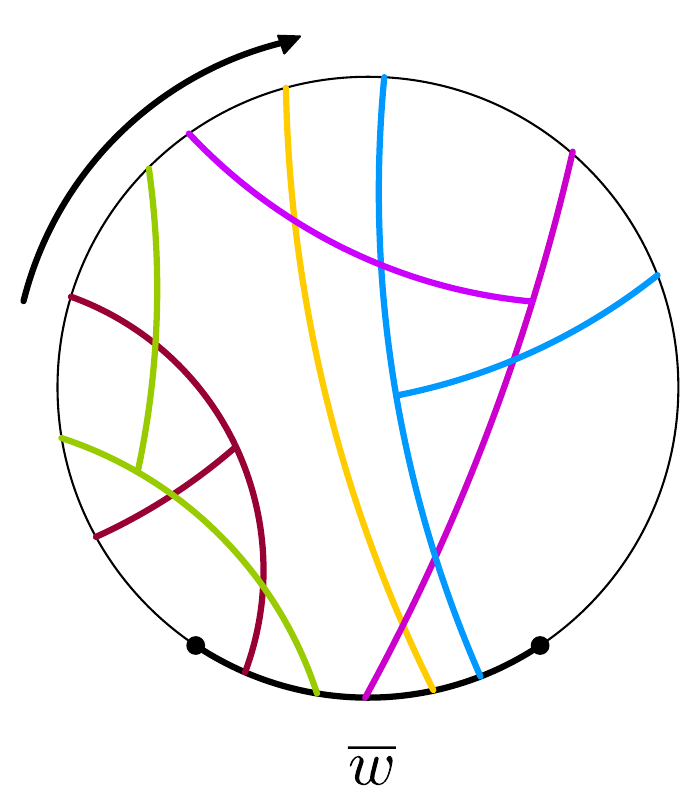}
  \caption{$w$-left-combed}
  \label{fig:LeftCombed}
\end{subfigure}%
\hfill
\begin{subfigure}[c]{.32\textwidth}
  \centering
  \includegraphics[scale=0.4]{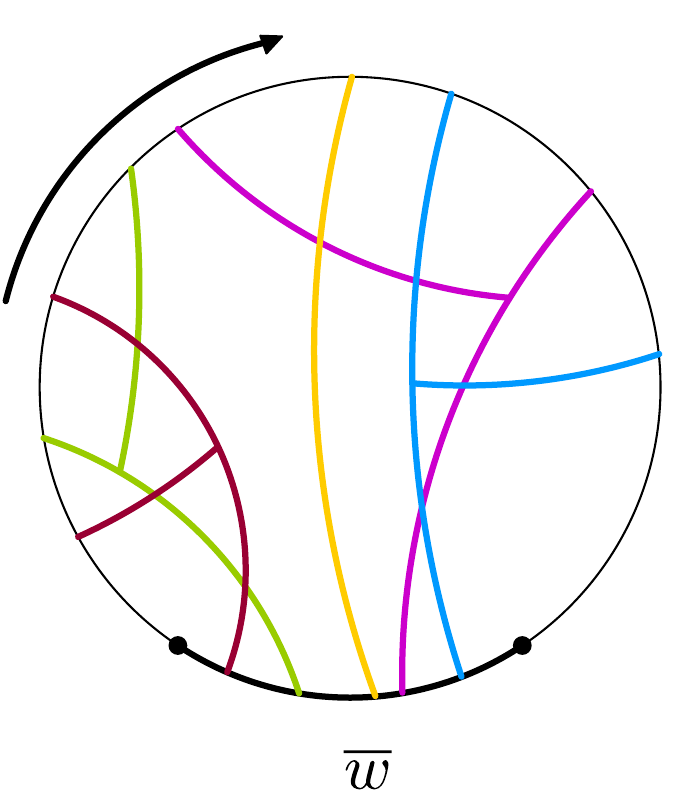}
  \caption{$w$-right-combed}
  \label{fig:RightCombed}
\end{subfigure}
\caption{Example of combed disk diagrams.} 
\label{fig:w-combed}
\end{figure}

For the proof of \Cref{thm:join-busting} we should notice that these definitions imply:

\begin{fact}\label{lem:combing_subsubwords}If $w'$ is a subword of a geodesic word $w$ along the boundary of a $w$--left-combed (or $w$--right-combed) disk diagram $\dd$, then $\dd$ is also $w'$--left-combed (or $w'$--right-combed, respectively).  
\end{fact} 
    
Before establishing the existence of left/right-combed disk diagrams (\Cref{lem:Combing} below), we record an elementary fact about disk diagrams and commutation.  Given a prism word $s$ that is equal to the identity, and a disk diagram $\mathcal{D}$ for $s$, if $a$ and $b$ are two successive letters in $s$ which belong to two adjacent vertex groups, then we define \emph{the commuting operation} on $\mathcal{D}$ as follow : If $a$ and $b$ belong to two adjacent vertex groups, then the square with sides labeled with $a,b,\bar a,\bar b$ is a Van Kampen diagram for $aba^{-1}b^{-1}$ that can directly be seen as a disk diagram. The commuting operation consists of gluing $\mathcal{D}$ to this square along $ab$, as in \Cref{lem:disk_concatenation}. This yields a disk diagram for the same word as $\mathcal{D}$ up to commutating $a$ and $b$.

In a certain sense, the commuting operation does not change beginning and ending functions.  To make this precise, recall that by \Cref{lem:finite_geodesics}, any two prism geodesics for $g$ differ only by the order of their letters.  Given such a prism geodesic $w=w_1\cdots w_m$, and a permutation $\sigma$ of $\{1,\ldots, m\}$ such that $w'=w_{\sigma(1)}\cdots w_{\sigma(m)}$ is also a prism-geodesic representative for $g$, we say that a disk-diagram $\dd$ for $g \overline{w}$ and a disk-diagram $\dd'$ for $g\overline{w}'$ have \emph{the same} beginning function (resp. ending function) if for all $i$, $b^{\dd'}(i)=b^{\dd}(\sigma(i))$ (resp. $e^{\dd'}(i)=e^{\dd}(\sigma(i))$) 

\begin{lem}\label{lem:commuting}
    Let $g$ be a prism word and $w$ a prism-geodesic representative. Given a disk diagram $\mathcal{D}$ for $g\overline{w}$, any commuting operation applied to pairs of adjacent letters of $w$ produces a disk-diagram with the same beginning and ending functions as $\dd$.
\end{lem}

\begin{proof}
    This is direct by construction. Indeed, as noted previously, when constructing a disk diagram by gluing, the new dual graphs are concatenations of dual graphs coming from the initial disk diagrams.
\end{proof}

\begin{lem}[Combing]
    \label{lem:Combing} 
    Given a prism word $g$, a prism-geodesic representative $w$ for $g$, and a disk-diagram $\dd$ for $g\overline{w}$, there exists a reordering of the letters of $w$ into a prism representative $w^{(l)}$ (resp. $w^{(r)}$) of $g$ and a $w^{(l)}$-left-combed disk-diagram $\mathcal{D}_l$ for $g\overline{w}^{(l)}$ (resp. a $w^{(r)}$-right-combed disk-diagram $\mathcal{D}_r$ for $g\overline{w}^{(r)}$) with the same beginning and ending function as $\mathcal{D}$.
\end{lem}

\begin{proof}
    Given $g$, we prove the lemma for the left-combed representative, as right-combing follows the same idea. 
    Start with any prism geodesic $w= w_1\cdots w_m$ for $g$, and fix a disk diagram $\mathcal{D}$ for the word $g\overline{w}$. For $1\leq i \leq n-1$, if $b^\mathcal{D}(i)>b^\mathcal{D}(i+1)$ then the graphs dual to $w_i$ and $w_{i+1}$ cross in $\mathcal{D}$, meaning that $w_i$ and $w_{i+1}$ are in adjacent vertex groups, and therefore commute. Hence $w_1\cdots w_{i-1}w_{i+1}w_iw_{i+2}\cdots w_m$ is still a geodesic word for $g$. \Cref{lem:commuting} ensures that the commuting operation produces a disk-diagram for $g\overline{w_1\cdots w_{i-1}w_{i+1}w_iw_{i+2}\cdots w_m}$ with the same ending and beginning functions as $\mathcal{D}$. The bubble sort algorithm can then be applied to conclude the proof.
\end{proof}

\begin{rem}
    Although we will not need it, it is worth noticing that the combing process does more than leaving the beginning and ending functions unchanged. If $w^{(l)}$ differs from $w$ by a permutation $\sigma$ of the letters of $w$, then for any $i$, $w^{(l)}_{i}$ is dual in $\dd^{(l)}$ to exactly the same letters of $g$ as $w_{\sigma(i)}$ in $\dd$.
\end{rem}

We record here an application of combing to geodesic representatives of words in a subgroup $H$, which provides quantitative control which will be important for the proof of Theorem \ref{thm:join-busting}.

\begin{lem}[Reduction control]\label{lem:RedCon}
Let $H$ be a finitely generated subgroup of $\gp$ with finite generating set $S_H$, and let $h = h_1\cdots h_n$ be an $S_H$--geodesic. Write each $h_i$ as an $S_p$--geodesic, so that $h$ also appears as a prism word. Let $s$ be an $S_p$--geodesic representative for $h$, and let $w=w_1\cdots w_m$ be a subword of $s$. Consider a $w$-left-combed disk diagram $\dd$ for the prism word $h\overline{s}$, and let $b^\dd$ be the beginning function of $\dd$ with respect to $w$.  Given $1 \leq \ell < r \leq m$, let $L$ and $R$ be the indices of the subwords $h_i$ of $h$ containing the $S_p$--letters indexed by $b^{\dd}(\ell)$ and $b^{\dd}(r)$ respectively. Then
    \begin{equation*}
        R-L\geqslant \frac{r-l}{\mH}, 
    \end{equation*}
    where $\mH$ is the maximal $S_p$--length for elements in $S_H$.
\end{lem}

\begin{proof}
    First note that, by left-combing, for any $\ell<k<r$ we have that $b^\dd(l)< b^\dd(k)<b^\dd(r)$.  Each graph dual to $w_k$ has its first root on some edge of a subword $h_i$ for $L\leq i \leq R$, and each such subword has at most $\mH$ edges; two of these $(R-L+1)\cdot\mH$ total possibilities are already dual to $w_\ell$ and $w_r$.  Therefore
    \begin{equation*}
        r-\ell-1\leq(R-L+1)\mH - 2 < (R-L+1)\mH - 1.
    \end{equation*}
\end{proof}

\subsection{Disk diagrams for almost star-free subgroups}
\label{sec:star-free}

We now apply disk diagrams to study finitely generated almost star-free subgroups of $\gp$. Given such a subgroup $H$ together with a finite generating set $S_H$, we always consider elements of $S_H$ as represented by $S_p$--geodesics. Hence if $h$ is an element of $H$ written as an $S_H$--geodesic $h_1\cdots h_n$, then it is also represented by a prism word, and if $w$ is a prism-geodesic representative for $h$, then a disk diagram for $h\overline{w}$ means a disk diagram for the associated prism word obtained by replacing the letters $h_i$ of $h$ by their $S_p$--geodesic representatives.  Recall from \Cref{lem:RedCon} that we denote by $\mH$ the maximum prism-length of an element in $S_H$.

The first of this series of lemmas gives a quantitative control on how $S_H$--geodesic words reduce to $S_p$--geodesic words.  It is adapted from Lemma 4.1 in \cite{RAAGstable}, a right-angled Artin group counterpart bounding ``cancellation diameter''; in RAAG disk diagrams, the analogue for dual graphs are arcs that represent the cancellation of a generator with its inverse, where these are letters from the (finite) standard generating set for the RAAG.  In our setting of graph products, dual graphs represent the combining of prism generators from the same vertex group to yield the identity (as observed in \Cref{lem:single_hyperplane}).

\begin{lem}[Bounded combination diameter]\label{lem:bounded-cancellation-diameter}
    Suppose $H$ is a finitely generated almost star-free subgroup of a graph product $\gp$, and let $S_H$ be a finite generating set of $H$.  Then there exists a constant $D=D(S_H)$ with the following property:  Suppose $h\in H$ is an $S_H$--geodesic word $h=h_1\cdots h_n$ equivalent to the $S_p$--geodesic word $w$.  Then in any disk diagram for the identity word $h\overline{w}$, if a dual graph has roots in $h_i$ and $h_j$, then $|i-j|<D$.
\end{lem}

\begin{proof}
    Let us suppose, by contradiction, that no such $D$ exists. This means that there exist sequences of $S_H$--geodesic and $S_p$--geodesic words
\begin{align*}
    h^{(k)} &= h_1^{(k)}\cdots h_{n(k)}^{(k)} &&\text{ where } h_i^{(k)}\in S_H \text{ and } n(k) \text{ is minimal, and} \\
    w^{(k)} &= h^{(k)} &&\text{ where }w^{(k)}\text{ is an $S_p$--geodesic};
\end{align*}
    along with disk diagrams $\mathcal{D}^{(k)}$ for the identity words $h^{(k)}\overline{w}^{(k)}$ (recall that each element of $S_H$ comes with a prism geodesic representative, allowing us to see $h^{(k)}$ as a prism word); and indices $i(k)<j(k)$ such that some dual graph $\alpha^{(k)}$ of $\mathcal{D}^{(k)}$ roots in letters of $h_{i(k)}^{(k)}$ and $h_{j(k)}^{(k)}$, where $j(k)-i(k)$ grows arbitrarily large.%\[\lim_{m\to\infty}j(m)-i(m) = \infty.\]

\begin{figure}[!h]
    \centering
    \includegraphics[width=0.5\linewidth]{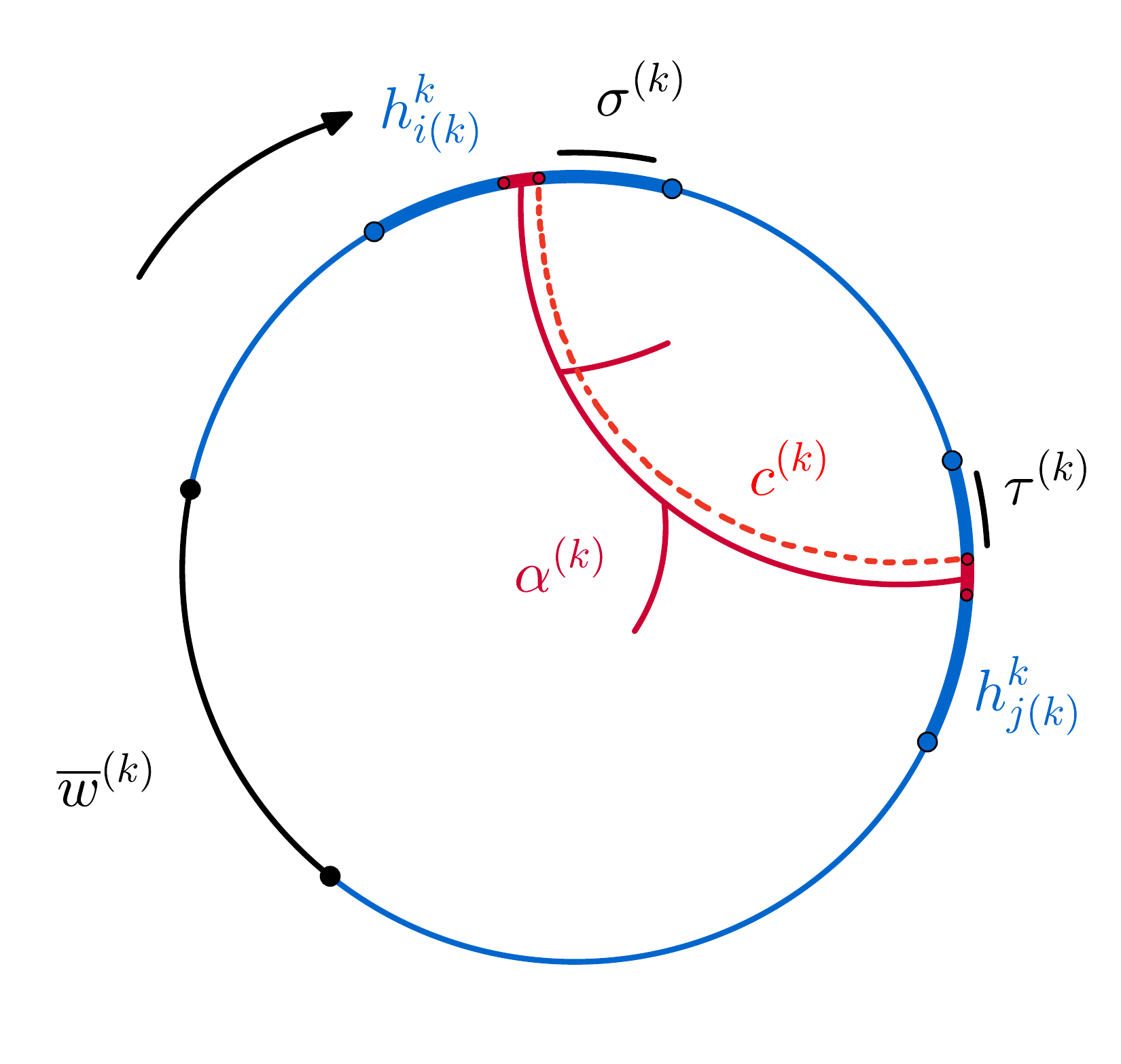}
    \caption{Decomposition of $h^{(k)}\overline{w}^{(k)}$}
    \label{fig:Lemma5.5}
\end{figure}

Let $v^{(k)}$ be the vertex corresponding to the dual graph $\alpha^{(k)}$.  Taking words along the carriers of each $\alpha^{(k)}$, connecting letters in $h_{i(k)}^{(k)}$ and $h_{j(k)}^{(k)}$ respectively, this sequence of disk diagrams implies the existence of a sequence of $c^{(k)} \in G_{\Star(v^{(k)})}$ that are equal to subwords of $h^{(k)}$ of the form
\begin{align*}
c^{(k)} &= \sigma^{(k)} h^{(k)}_{i(k)+1}\cdots h^{(k)}_{j(k)-1}\tau^{(k)}\text{, where}\\
\sigma^{(k)} &\text{ is a suffix of }h^{(k)}_{i(k)}\text{ and}\\
\tau^{(k)} &\text{ is a prefix of }h^{(k)}_{j(k)}\text{.}
\end{align*}
This construction is illustrated in Figure \ref{fig:Lemma5.5}.

Because $\G$ is finite, we may pass to a subsequence so that $c^{(k)}$ are all in the same $G_{\Star(v)}$.  Because Lemma \ref{lem:finite_geodesics} bounds the number of prism-geodesic ways to write the finitely many elements of $S_H$, we can pass to a further subsequence so that $\sigma^{(k)}$ and $\tau^{(k)}$ are also constant, equal to $\sigma$ and $\tau$ respectively.

In this subsequence, $c^{(k)} = \sigma h'^{(k)} \tau$ where $h'^{(k)} \in H$, and $|h'^{(k)}|_{S_H} = j(k) - i(k) - 1$ grows arbitrarily large. We can therefore pass to a further subsequence such that $|h'^{(k)}|_{S_H} \neq |h'^{(l)}|_{S_H}$ for all $k \neq l$, and so $h'^{(k)} \neq h'^{(l)}$ for all $k \neq l$. As a result, we have that \[h'^{(1)}(h'^{(k)})^{-1} = \sigma^{-1}c^{(1)}(c^{(k)})^{-1}\sigma \in \sigma^{-1}G_{\Star(v)}\sigma\] is an infinite sequence of distinct elements of $H$, so $H$ is not almost star-free.
\end{proof}

Like the previous lemma, the following lemma mirrors a right-angled Artin group version in \cite{RAAGstable}, this time of the same name.  This concerns \emph{vanishing subwords} for $S_H$--geodesic words $h$, defined as follows.  Suppose $h = h_1\cdots h_n$ where $h_i \in S_H$ are $S_p$--geodesic words and $n$ is minimal.  We say the subword $h_i\cdots h_j$ is \emph{vanishing} if there exists an $S_p$--geodesic word $w$ equal to $h$ with a disk diagram for $h\overline{w}$ where no dual graph rooted in $\overline{w}$ intersects $h_i \cdots h_j$.  In this sense, the subword $h_i\cdots h_j$ does not contribute to the letters in the prism-geodesic word $w$.

\begin{lem}[Bounded non-contribution]\label{lem:bounded-non-contribution}
    Suppose $H$ is a finitely generated almost star-free subgroup of a graph product $\gp$, and let $S_H$ be a finite generating set of $H$. Then there exists a constant $K = K(S_H)$ such that if $h_i \cdots h_j$ is a vanishing subword of $h \in H$ as in the definition above, then $j-i < K$.
\end{lem}

\begin{proof}
    Let $h$ be a $S_H$--geodesic with length $n$ in $S_H$. Consider a vanishing subword $h' = h_i\cdots h_j$ of $h$ in a disk diagram $\mathcal{D}$ for $h\overline{w}$ where $w$ is an $S_p$--geodesic. Let $D=D(S_H)$ be as in \Cref{lem:bounded-cancellation-diameter}. Let $\tau$ be the empty word if $i=1$, and otherwise the subword $h_{i'}\cdots h_{i-1}$ of $h$, where $i' = \max\{1, i-D\}$. Likewise, let $\sigma$ be the the empty word if $j=n$, and otherwise $h_{j+1} \cdots h_{j'}$ where $j' = \min\{n,j+D\}$. Note that, by \Cref{lem:bounded-cancellation-diameter}, and since $h'$ is vanishing in $\dd$, any dual graph in $\mathcal{D}$ rooted in $h'$ can only have endpoints in $\tau h'\sigma$, which is an $S_H$--geodesic subword of $h$.
 
    We first want to show that $h'$ is also a vanishing subword for $\tau h'\sigma$. Let $w'$ be a prism-geodesic representative for $\tau h'\sigma$; we will find a disk diagram for $\tau h'\sigma\overline{w'}$ that evidences this vanishing of $h'$.
\begin{figure}[!h]
    \centering
    \includegraphics[width=0.5\linewidth]{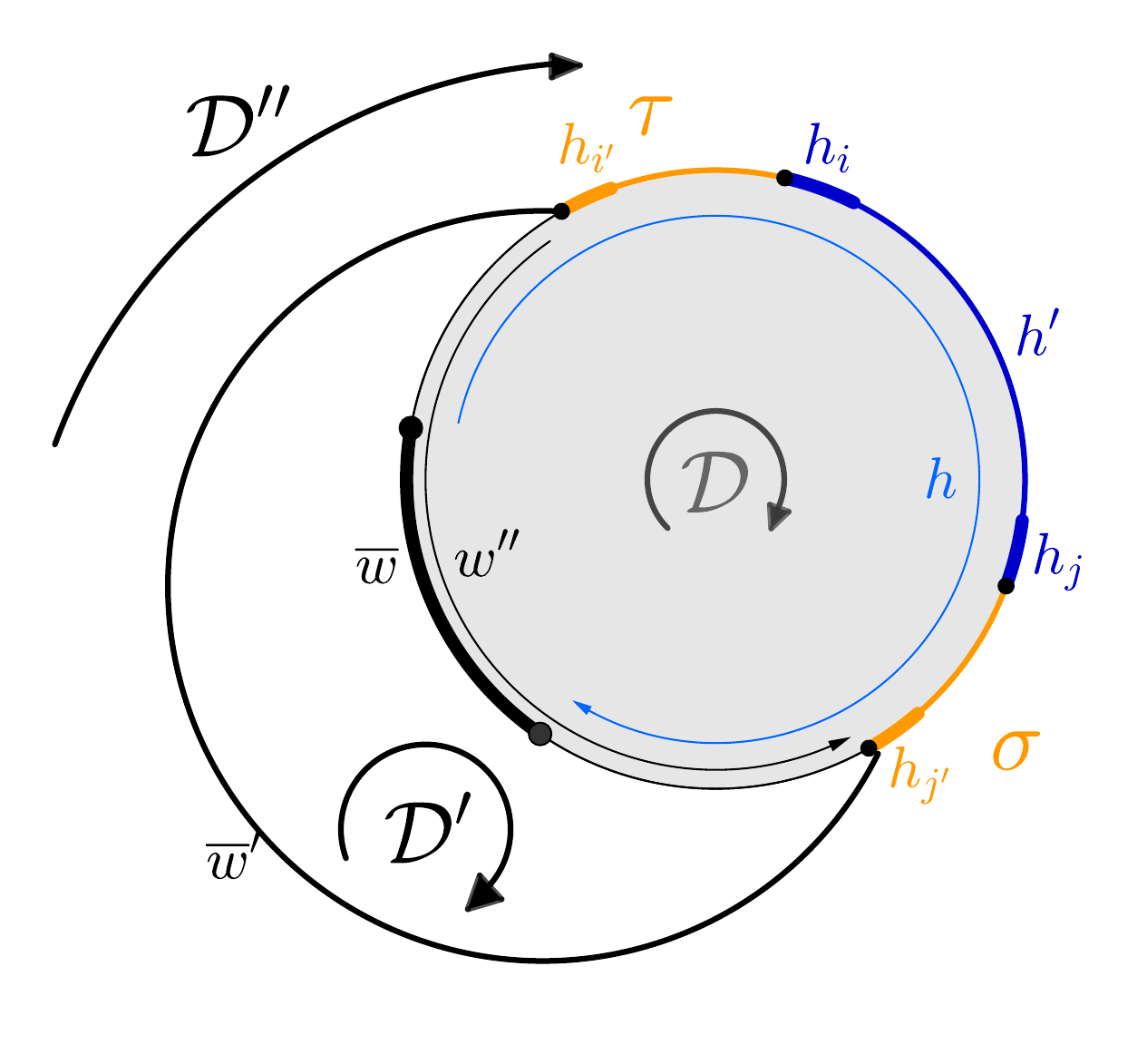}
    \caption{Creation of a disk diagram $\mathcal{D}''$ where $h'$ vanishes}
    \label{fig:Lemma5.6_1}
\end{figure}
Let $w'' = \overline{h_{i'-1}\cdots h_1}w\overline{h_n\cdots h_{j'+1}}$, without the first portion if either $i=1$ or $i'=1$, and without the last portion if either $j=n$ or $j'=n$. This means that $w'$ is also a prism-geodesic representative for $w''$. Let $\mathcal{D}'$ be a disk diagram for $w''\overline{w}'$, and let $\mathcal{D}''$ be the concatenation of $\mathcal{D}$ and $\dd'$ along $w''$, see \Cref{fig:Lemma5.6_1}.  By \Cref{lem:disk_concatenation} this is a disk diagram for $\tau h' \sigma \overline{w}'$. Because the dual graphs in $\mathcal{D}''$ are concatenations of dual graphs in $\mathcal{D}$ and $\mathcal{D}'$, no graph dual to $w'$ in $\mathcal{D}''$ roots in $h'$; thus $h'$ is vanishing for $\tau h'\sigma$ as claimed.

To complete the proof, it suffices to show that there are only finitely many possibilities for $h'= \overline{\tau}w'\overline{\sigma}$, since this automatically bounds $|h'|_H = j-i+1$. First note that the $S_H$--lengths of $\tau$ and $\sigma$ are bounded by $D$, so as $S_H$ is finite, there are only finitely many possibilities for $\overline{\tau}$ and $\overline{\sigma}$. By \Cref{lem:single_hyperplane}, we have that each edge labeled $a$ in $w'$ satisfies $a=a_1\dots a_m$, where the $a_i$ are the roots in either $\tau$ or $\sigma$ of the graph dual to $a$, in order along the boundary. As the $S_H$--lengths of $\tau$ and $\sigma$ are bounded by $D$, and \Cref{lem:finite_geodesics} bounds the number of prism-geodesic ways to write the finitely many elements of $S_H$, this gives only finitely many possibilities for each prism letter $a$ of $w'$. As the $S_p$--lengths of $\tau$ and $\sigma$ are bounded by $D\mH$, the $S_p$--length of $w'$ is bounded by $2D\mH$. Therefore there are only finitely many possibilities for $w'$, and likewise only finitely many possibilities for $h'$.
\end{proof}

 In contrast to the previous two results, Lemma \ref{lem:bounded-concatenation} addresses new behavior arising in disk diagrams for graph products.  Because disk diagram concatenation also concatenates dual graphs, we must verify we do not lose the quantitative control from \Cref{lem:bounded-cancellation-diameter}.

\begin{lem}[Bounded post-concatenation combination]
    \label{lem:bounded-concatenation}
    Suppose $H$ is a finitely generated almost star-free subgroup of a graph product $\gp$, and let $S_H$ be a finite generating set of $H$. Then there exists a constant $C=C(S_H)$ with the following property.  Suppose $h\in H$ is an $S_H$--geodesic word $h=h_1\cdots h_n$ equivalent to the $S_p$--geodesic word $w$, and that the sub-word $h'=h_i\cdots h_j$ of $h$ is equivalent to the $S_p$--geodesic word $u$. Let $\mathcal{D}$ be a disk diagram for the identity word $h\overline{w}$, let $\mathcal{D}'$ be a disk diagram for the identity word $\bar h'u$, and let $\mathcal{D}''$ be the concatenation of $\mathcal{D}$ and $\mathcal{D}'$ along $h'$. If a dual graph in $\mathcal{D}''$ has roots in $u$, $h_1\cdots h_{i-1}$, and $h_{j+1}\cdots h_n$, then $|i-j|<C$.
\end{lem}

\begin{proof}
    Consider the restriction of the relevant dual graph in $\dd''$ to $\mathcal{D}$. This gives us a dual graph in $\mathcal{D}$ that is rooted in both $h_1\cdots h_{i-1}$ and some $h_{i'}$ in $h'$, and a dual graph in $\mathcal{D}$ that is rooted in both $h_{j+1}\cdots h_n$ and some $h_{j'}$ in $h'$ (possibly the same graph). We can assume without loss of generality that $i'\leqslant j'$. 
    By \Cref{lem:bounded-cancellation-diameter} applied to $\mathcal{D}$, we therefore have that $|i-i'|<D$ and $|j-j'|<D$, where $D$ is the constant from \Cref{lem:bounded-cancellation-diameter}.

    Let $v$ be the vertex associated to the relevant dual graph in $\mathcal{D''}$. By taking a word $W$ along the carrier of this dual graph, we can see that $h'$ can be written as $\sigma W \tau$, where $\sigma$ is an $S_p$--subword of $h_i\cdots h_{i'}$, $\tau$ is an $S_p$--subword of $h_{j'}\cdots h_j$, and $W\in G_{\Star(v)}$, as illustrated in \Cref{fig:Bounded concatenation}.

     \begin{figure}[!h]
    \centering
    \includegraphics[width=0.4\linewidth]{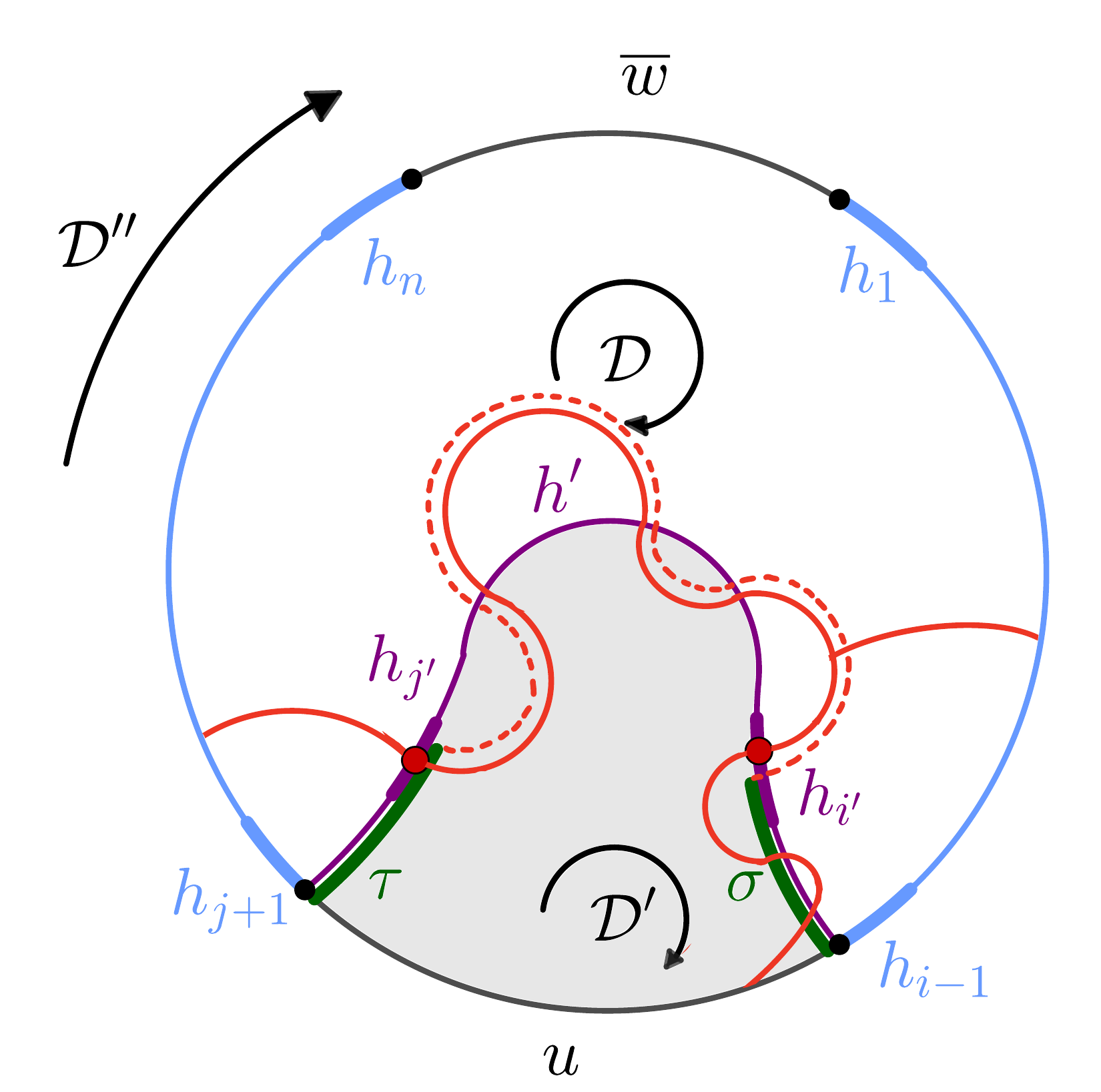}
    \caption{The relevant dual graph in $\dd\cup\dd'$.}
    \label{fig:Bounded concatenation}
\end{figure}

    Now suppose, by contradiction, that no bound on $|i-j|$ exists. We can therefore pick a sequence of $h^{(k)}$ and $h'^{(k)}$ in $H$ where the $h'^{(k)}$ have strictly increasing $S_H$--length, and such that disk diagrams and dual graphs exist for each which satisfy the hypotheses of this lemma. In particular, we can write $h'^{(k)}$ as $\sigma^{(k)} W^{(k)} \tau^{(k)}$, where $\sigma^{(k)}$ is an $S_p$-subword of $h_{i(k)}^{(k)}\cdots h_{i'(k)}^{(k)}$, $\tau^{(k)}$ is an $S_p$--subword of $h_{j'(k)}^{(k)}\cdots h_{j(k)}^{(k)}$, and $W^{(k)}\in G_{\Star(v^{(k)})}$. As there are only finitely many possibilities for $v^{(k)}$, and both $|i(k)-i'(k)|<D$ and $|j(k)-j'(k)|<D$, in the same way as in previous lemmas in this section, we can pass to a subsequence such that $h'^{(k)}=\sigma W^{(k)}\tau$ for some constant words $\sigma$ and $\tau$, and such that each $W^{(k)}\in G_{\Star(v)}$ for the same vertex $v$.
    
    Recall that the $S_H$--length of the $h'^{(k)}$ strictly increase as $k$ increases, so in particular no two are equal as elements of $H$. We therefore have that
    \begin{equation*}
        h'^{(1)}(h'^{(k)})^{-1}=\sigma W^{(1)}(W^{(k)})^{-1}\sigma^{-1}\in \sigma G_{\Star(v)}\sigma^{-1}
    \end{equation*}
    is an infinite sequence of distinct elements of $H$. This contradicts $H$ being almost star-free.
\end{proof}

\subsection{Join-busting}
\label{sec:proof}

In this section we prove a key property of almost join-free subgroups of $\gp$; this will make use of the disk diagram tools established thus far.

\begin{defn}\label{def:joinbusting}
A \emph{join (sub)word} is a (sub)word consisting of letters from $S_p$ such that all these letters belong to the same join subgroup of $\gp$.  A subgroup $H$ of a graph product $\gp$ is said to be \emph{N--join-busting} for some $N>0$ if for any prism geodesic $w$ representing an element $h\in H$, the prism length of any join subword of $w$ is bounded by $N$.
\end{defn}

We have one more technical result, \Cref{lem:homogeneity}, to record before we proceed to the main theorem of this section.

%%%%%%%%%%%%%%%%%
\begin{defn}
    We call an $S_p$--word \emph{M--homogeneous with respect to $\Lambda$} for some subgraph $\Lambda$ of $\Gamma$, if all of its letters are contained in $G_{\Lambda}$, and all of its subwords of prism length at least $M$ contain at least one letter from each vertex group of $\Lambda$.   
\end{defn}

\begin{lem}\label{lem:homogeneity} 
    If $(u_k)_{k\in\N}$ is a sequence of $S_p$--geodesic join words, such that $|u_k|_p$ tends to infinity with $k$, then there exists a sequence of $S_p$--geodesic words $(w_n)_{n\in\N}$ such that:
    \begin{itemize}
        \item the $w_n$ are subwords of a subsequence of $u_k$,
        \item the lengths $|w_n|_p$ tend to infinity with $n$,%$|w_n|_p\xrightarrow[n\to\infty]{}\infty$,
        \item there exists $M>0$ and a join ${\Lambda}$ such that the $w_n$ are all $M$--homogeneous with respect to the same subgraph of ${\Lambda}$.
    \end{itemize}
\end{lem}

\begin{proof}
    This proof works in the same way as \cite[Lemma 5.1]{RAAGstable}.
\end{proof}

\begin{thm}[Join-busting]\label{thm:join-busting} Let $\G$ be a graph with no isolated vertices. Suppose $H$ is a finitely generated almost join-free subgroup of a graph product $\gp$, and let $S_H$ be a finite generating set of $H$. Then there exists a constant $N=N(S_H)$ such that $H$ is N--join-busting.
\end{thm}

\begin{proof}
    We begin by noticing that as $\G$ has no isolated vertices, the assumption that $H$ is almost join-free implies that $H$ is also almost star-free. This allows us to apply the results from \Cref{sec:star-free}.
    
    Towards contradiction, assume that $H$ is not $N$--join-busting for any $N$.  Then we can find a sequence $(h^ {(k)})_{k\in\N}$ of elements in $H$ such that some prism-geodesic representatives can be written $a^{(k)}w^{(k)}b^{(k)}$, where the $w^{(k)}$ are join words with $\displaystyle \lim_{k\to\infty}|w^{(k)}|_p=\infty$.  For each $k$, we fix an $S_H$--geodesic representative \[h^{(k)}=h^{(k)}_1\cdots h^{(k)}_{n(k)},\] for $h^{(k)}$, meaning $n(k)$ is the length of $h^{(k)}$ in $S_H$ and each $h^{(k)}_i$ is an element of $S_H$.  We consider $h^{(k)}$ as a word in $S_p$ by representing each $h^{(k)}_i$ by a prism-geodesic word in $S_p$.
    
    By \Cref{lem:Combing}, we can assume that we have a $w^{(k)}$--left-combed disk diagram $\dd^{(k)}$ for each identity word $h^{(k)}\overline{a^{(k)}w^{(k)}b^{(k)}}$, after re-ordering the letters of $w^{(k)}$.  Applying Fact \ref{lem:combing_subsubwords} and Lemma \ref{lem:homogeneity}, we may also assume that there exists $M>0$ such that the $w^{(k)}$ are all $M$--homogeneous with respect to the same subgraph $\Lambda$ of $\G$, where $\Lambda$ is contained in a join.  That is, we pass to a subsequence and replace $w^{(k)}$ by subwords, effectively enlarging the parts of the prism-geodesic representative for $h^{(k)}$ labeled by $a^{(k)}$ and $b^{(k)}$.

    By Lemma \ref{lem:rooting}, each dual graph in $\dd^{(k)}$ rooted in $w^{(k)}$ has its remaining roots only in $h^{(k)}$.  Our goal now is to find an $S_H$--subword of $h^{(k)}$ whose prism geodesic representative is of the form $\sigma^{(k)}W^{(k)}\tau^{(k)}$, such that $W^{(k)}$ is a join word whose $S_p$--length tends to infinity with $k$, whereas $\sigma^{(k)}$ and $\tau^{(k)}$ each range over finitely many possibilities.

    Towards this, first let us reduce notation by fixing $k$.  We set $h = h^{(k)}=h_1\cdots h_n$, $a = a^{(k)}$, $b = b^{(k)}$, $\dd=\dd^{(k)}$, and $w=w^{(k)}=w_1\cdots w_m$. Assume without loss of generality that $m>2M$. Since $\dd$ is $w$--left-combed, we may apply \Cref{lem:RedCon} for $\ell = M$ and $r = m - M$.  Thus, where $h_L$ and $h_R$ contain the first edges dual to the same graphs as $w_\ell$ and $w_r$ respectively, we have by \Cref{lem:RedCon}
    \[R - L \geq \frac{m-2M}{\mH}.\]
    In particular, longer $S_p$--length of $w$ implies longer $S_H$--length of $h_{L+1}\cdots h_{R-1}$; denote the latter word by $h'$.  Henceforth we assume $w$ is long enough so that $R-L > C+2D$, where $D$ and $C$ are the constants taken from \Cref{lem:bounded-cancellation-diameter} and \Cref{lem:bounded-concatenation}. We assume without loss of generality that $D$ is an integer, and that $D\geqslant 2$.
    
 Let $x$ be a prism letter of $h'$ whose dual graph in $\mathcal{D}$ has a root in $awb$. Either it is rooted in $w$, in which case $x$ belongs to $G_\Lambda$, or it is rooted in one of $a$ or $b$. Since $\dd$ is $w$--left-combed, all the dual graphs rooted in $w_i$ for $1\leq i \leq M$ have roots in $h_1\cdots h_R$, so a graph dual to $x$ rooted in $a$ will cross all of them. Similarly, a graph dual to $x$ rooted in $b$ will cross all the dual graphs rooted in $w_i$ for $m-M\leq i \leq m$. In either case, because $w$ is $M$--homogeneous for $\Lambda$, the graph dual to $x$ crosses a graph rooted in an edge corresponding to each vertex group of $\Lambda$; therefore $x$ belongs to $G_{\Star(\Lambda)}$.  Thus we have:
    
    \begin{fact}\label{fact:rootawb}
    If $x$ is a letter of $h'$ with dual graph rooted in $awb$, then $x \in G_{\Star(\Lambda)}$.
    \end{fact}
    
    Now we consider the subword $h'' = h_{L+D}\cdots h_{R-D}$, and an $S_p$--geodesic representative $u=u_1\cdots u_t$, where $u_i \in S_p$ and $t = |h''|_p$. By assumption on $R-L$, the $S_H$--length of $h''$ is positive. Note also that, since $S_H$ is finite, we have finitely many possibilities for the subwords $h_{L+1} \cdots h_{L+D-1}$ and $h_{R-D+1}\cdots h_{R-1}$.  We also note that, by \Cref{lem:bounded-cancellation-diameter}, any graph dual to a letter of $h''$ can only have roots in $h'$ or $awb$.

    Let $\mathcal{D}''$ be a disk-diagram for $\overline{h}''u$, and concatenate it with $\mathcal{D}$ to obtain a disk diagram $\mathcal{P}$ for the word $uh_{R-D+1}\cdots h_n\overline{awb}h_1\cdots h_{L+D-1}$. This is the situation illustrated in \Cref{fig:JBConc}. By construction, and by \Cref{fact:concatenate-dual-graph} and \Cref{lem:rooting}, a graph dual to a letter of $u$ in $\mathcal{P}$ has its remaining roots in some combination of $h_{L+1} \cdots h_{L+D-1}$, $h_{R-D+1}\cdots h_{R-1}$, and $awb$. In particular, by \Cref{lem:bounded-concatenation}, and our assumption that $R-L > C+2D$, any graph dual to a letter in a $u$ can only have roots in exactly one of $h_{L+1} \cdots h_{L+D-1}$ or $h_{R-D+1}\cdots h_{R-1}$. As a consequence, if a graph dual in $\mathcal{P}$ to a letter of $u$ has a root in $awb$ (it cannot have more than one by \Cref{lem:rooting}), then it should be its beginning or ending one (in the sense of beginning and ending functions with respect to $u$, as defined in \Cref{sec:combing}). Moreover, by \Cref{fact:concatenate-dual-graph} and \Cref{fact:rootawb}, any letter of $u$ whose dual graph has a root in $awb$ must belong to $x \in G_{\Star(\Lambda)}$.

    \begin{figure}[!h]
\begin{subfigure}[c]{.32\textwidth}
  \centering
  \includegraphics[scale=0.26]{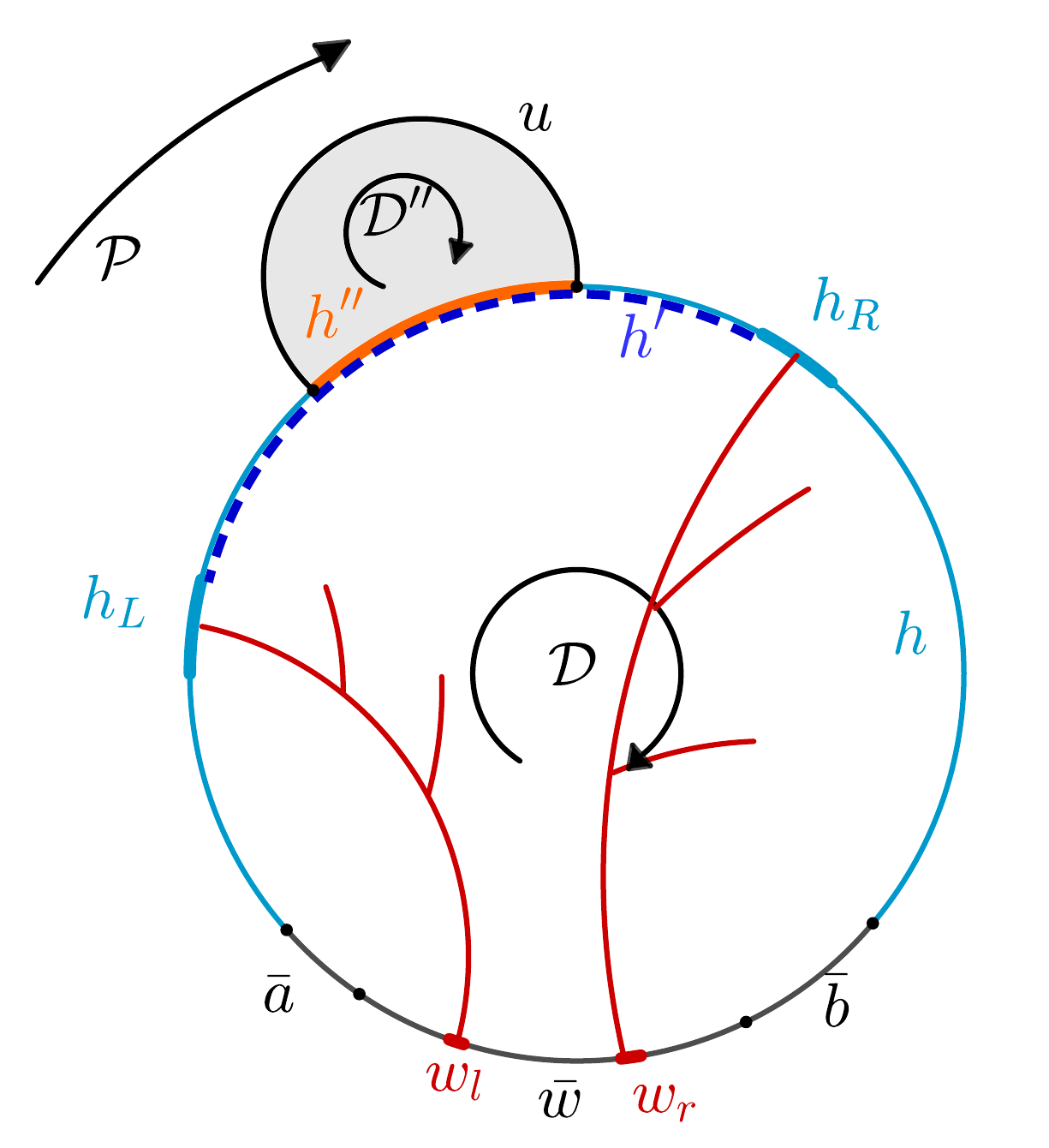}
  \caption{Concatenation }
  \label{fig:JBConc}
\end{subfigure}%
\hfill
\begin{subfigure}[c]{.32\textwidth}
  \centering
  \includegraphics[scale=0.26]{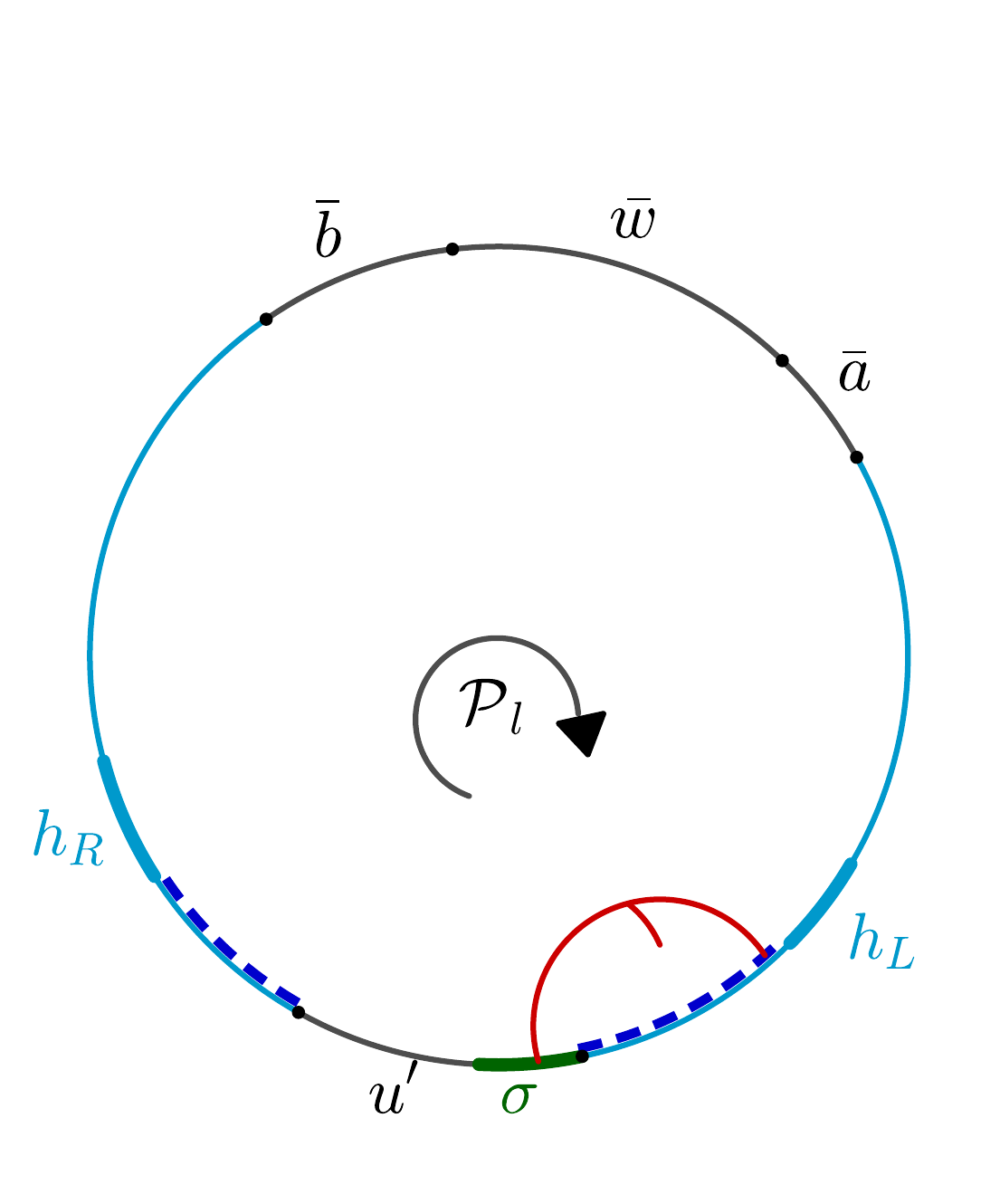}
  \caption{Left-combing}
  \label{fig:JBLeftComb}
\end{subfigure}%
\hfill
\begin{subfigure}[c]{.32\textwidth}
  \centering
  \includegraphics[scale=0.26]{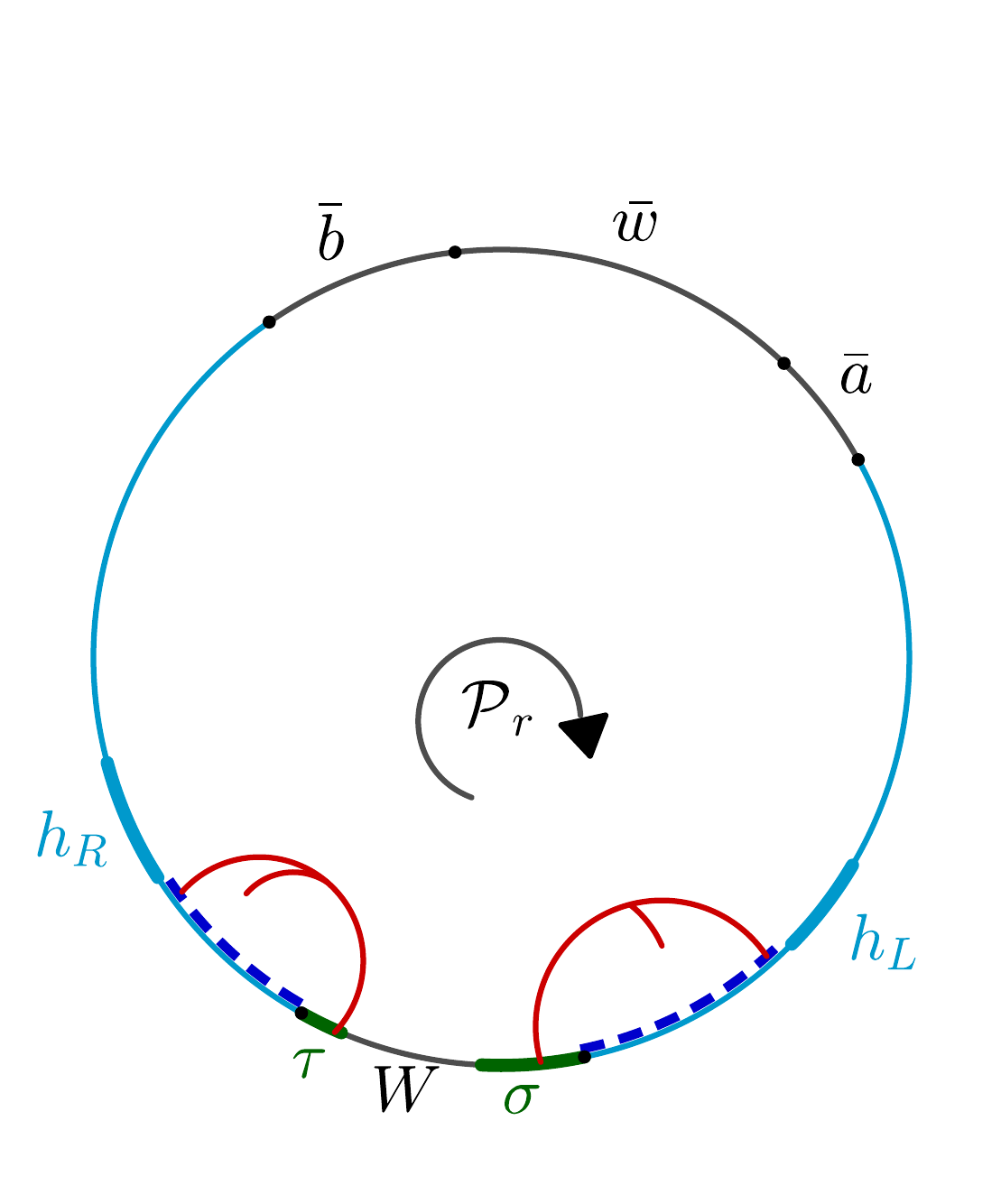}
  \caption{Right-combing}
  \label{fig:JBRightComb}
\end{subfigure}
\caption{Illustration of the different steps of the proof}
\label{fig:JoinBusting}
\end{figure}

Applying \Cref{lem:Combing} to $\mathcal{P}$ produces a prism-geodesic $u^{(l)}$ for the word $h''$ and a disk-diagram $\mathcal{P}_l$ for $u^{(l)} h_{R-D+1}\cdots h_n\overline{awb}h_1\cdots h_{L+D-1}$ which is $\overline{u}^{(l)}$--left-combed and has the same beginning and ending functions with respect to $u^{(l)}$ as $\mathcal{P}$ with respect to $u$. In particular, $u^{(l)}$ can be written $\sigma u'$, where the subword $\sigma$ is dual to graphs that only root in $h_{L+1}\cdots h_{L+D-1}$ (and no letter of $u'$ has this property) and thus has prism length at most $D\mH$. Applying \Cref{lem:Combing} to $\mathcal{P}_l$ allows us to obtain a $\overline{u}^{(r)}$--right-combed disk diagram $\mathcal{P}_r$ for $\sigma u^{(r)} h_{R-D+1}\cdots h_n\overline{awb}h_1\cdots h_{L+D-1}$, where $u^{(r)}$ is a prism-geodesic representative for $u'$ that is written in the form $W\tau$, where the graphs dual to $\tau$ only root in $h_{R-D+1}\cdots h_{R-1}$ (and no letter of $W$ has this property) and is also of prism-length at most $D\mH$. Therefore we have a prism-geodesic representative for $h''$ of the form $\sigma W \tau$, where the word $W$ contains only $u_i$ with dual graphs that have either their beginning or ending root in $awb$. As the beginning and ending roots are not affected by combing, we can conclude that $u_i\in G_{\Star(\Lambda)}$.

 Moreover, there are only finitely many possibilities for $\sigma$ and $\tau$. Indeed, consider a letter $u_j$ of $\sigma$, by \Cref{lem:single_hyperplane} we have that $u_j = s_1\cdots s_p$ where the $s_i$ are the letters on which the graph dual to $u_j$ roots in $\mathcal{P}$ and $p\leq D\mH$. Then the $s_i$ are letters in the prism-geodesic expression of the elements of $S_H$. Since $S_H$ is finite, and we fixed prism-geodesic representatives for its elements at the beginning, we have finitely many possibilities for each $u_j$. Hence, as $\sigma$ is of bounded prism-length, there are also finitely many possibilities for $\sigma$. The same argument applies to $\tau$.

    Applying the above process in all the $\mathcal{P}=\mathcal{P}^{(k)}$, we get a sequence \[ \sigma^{(k)}W^{(k)}\tau^{(k)}= (h'')^{(k)}\in H\] such that $W^{(k)} \in G_{\Star(\Lambda)}$. Because $\sigma^{(k)}$ and $\tau^{(k)}$ each range over finitely many possibilities, we pass to a subsequence $(h'')^{(k)}= \sigma W^{(k)} \tau$.  Note that the length of $(h'')^{(k)}$ grows with the the length of $(h')^{(k)}$, which itself grows with the length of $w^{(k)}$. Therefore there is $j_0\in\N$ and infinitely many $j\neq j_0$ such that the 
    \[(h'')^{(j_0)}((h'')^{(j)})^{-1}= \sigma W^{(1)}(W^{(j)})^{-1}\sigma^{-1}\]
    are distinct and non-trivial elements of $H$ conjugate by the same element of $\gp$ into $G_{\Star(\Lambda)}$. As $\Lambda$ is a subgraph of a join, by \Cref{lem:StarOfASubgraph} we have that $\Star(\Lambda)$ is a join, so this contradicts $H$ being almost join-free.
\end{proof}

\begin{rem}
    Note that if $\Lambda$ were a subgraph of a star, then $\Star(\Lambda)$ would not necessarily be a subgraph of a star. This is the key reason why we need to use almost join-free and join-busting in the statement of \Cref{thm:join-busting}, rather than the star equivalents.
\end{rem}

%%%%%%%%%%%%%%%%%%%%%%%%%%%%%%%%%%%%%%%
\section{Proof of Theorem 1.2}
\label{sec:main_proof}
%%%%%%%%%%%%%%%%%%%%%%%%%%%%%%%%%%%%%%%

This section is dedicated to proving the main theorem of this paper. Beforehand, let us recall the statement.

\begin{named}{\Cref{thm-intro:recognizing-stable}} 
        Let $\Gamma$ be a finite simple graph with no isolated vertices, $\mathcal{G}=\set{G_v}{v\in V(\Gamma)}$ be a collection of finitely generated infinite groups, and $P(\gp )$ be the prism complex of the graph product $\gp$. Let $H$ be a finitely generated subgroup of the graph product $\gp$. 
    Then the following are equivalent.
    \begin{enumerate}
        \item\label{thm-part:recognizing-stable-stable} $H$ is stable in $\gp$. 
        \item\label{thm-part:recognizing-stable-qi} The orbit of $H$ into the contact graph $\CX$ is a quasi-isometric embedding.
        \item\label{thm-part:recognizing-stable-ploxo} $H$ is almost join-free. 
    \end{enumerate}
\end{named}

We will prove $(\ref{thm-part:recognizing-stable-ploxo}) \implies (\ref{thm-part:recognizing-stable-qi})$ in \Cref{sec:almostjoinfree->qi}, $(\ref{thm-part:recognizing-stable-stable}) \implies (\ref{thm-part:recognizing-stable-ploxo})$ in \Cref{sec:stability}, and $(\ref{thm-part:recognizing-stable-qi}) \implies (\ref{thm-part:recognizing-stable-stable})$ in \Cref{sec:qi->stable}.

\subsection{Almost join-free subgroups embed quasi-isometrically into the contact graph}\label{sec:almostjoinfree->qi}

We start the proof of \Cref{thm-intro:recognizing-stable} by proving $(\ref{thm-part:recognizing-stable-ploxo}) \implies (\ref{thm-part:recognizing-stable-qi})$. Relying on \Cref{lem:Contact=star}, it is enough to show that $H$ quasi-isometrically embeds into $\gp$ for the star metric. To do so, we will use \Cref{thm:join-busting} together with the next two results.

First of all, recall that given a finitely generated subgroup $H$ of $\gp$ together with a finite generating set $S_H$, we implicitly consider all the elements in $S_H$ to be written as $S_p$--geodesic words, and we denote the maximum $S_p$--length of such elements by $\mH$. 

 \begin{prop}\label{prop:quasi-isometric embedding} If $H$ is a finitely generated almost star-free subgroup of a graph product $\gp$, then $H$ is quasi-isometrically embedded in the prism Cayley graph of $\gp$.
\end{prop}

\begin{proof} The proof follows \cite[Proposition 4.4]{RAAGstable}. Recall that $|\cdot|_H$ is the $S_H$--length of an element in $H$, and $|\cdot|_p$ the prism-length of an element in $\gp$. By definition of $\mH$, for every $h\in H$ we have that $|h|_H\geq |h|_{p}/\mH$.  Now consider a prism-geodesic word $w$ for $h$, and a disk diagram for $h\overline{w}$. By \Cref{lem:rooting}, any dual graph with a root in $\overline{w}$ has its remaining roots in $h$. By \Cref{lem:bounded-cancellation-diameter} each graph dual to a letter of $\overline{w}$ has its extremal roots in $h$ at $S_H$--distance at most $D$, and \Cref{lem:bounded-non-contribution} ensures there is no vanishing subword of $h$ of $S_H$--length greater than $K$. We can see that $|h|_H$ is maximal relative to $|h|_p=|w|_p$ when none of the dual graphs rooted in both $h$ and $w$ cross each other, all of the roots in $h$ of each of these dual graphs are as far apart as possible, and the roots of different dual graphs are separated by maximally large vanishing subwords, as illustrated in \Cref{fig:Lemma6.1}.  This implies that $|h|_H\leq (K+D)|h|_p + K$.
\end{proof}

\begin{figure}[!h]
    \centering
    \includegraphics[width=0.6\linewidth]{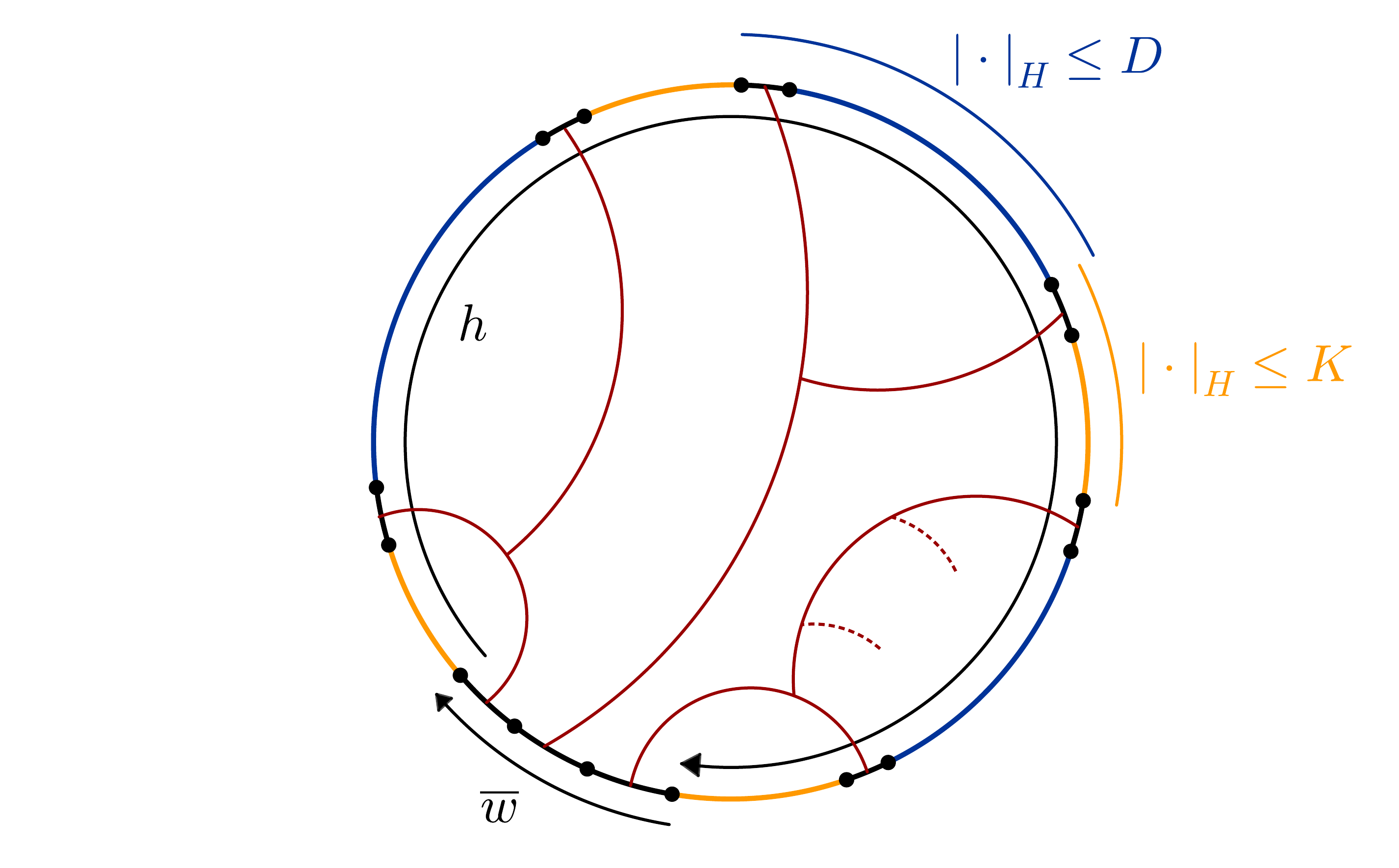}
    \caption{Hyperplane distribution in a disk diagram for $h\overline{w}$ with greatest difference between $|h|_H$ and $|h|_p$.}
    \label{fig:Lemma6.1}
\end{figure}

\begin{lem} \label{lem:StarAndPrismGeodesic} Let $g$ be a non-trivial element of a graph product $\gp$. There is a star-geodesic word $s_1\cdots s_n$ for $g$ whose letters admit prism-geodesic representatives $s_i=w_{i,1}\cdots w_{i,j(i)}$ such that their concatenation $w_{1,1}\cdots w_{n,j(n)}$ is a prism-geodesic representative for $g$.    
\end{lem}

\begin{proof}
    Fix a star-geodesic word $s_1\cdots s_n$ for $g$, and prism-geodesic words $w_{i,1}\cdots w_{i,j(i)}$ for each of the $s_i$. For all $i$ let $v_i\in V(\G)$ be such that $s_i\in G_{\Star(v_i)}$, and for $1\leqslant \ell\leq j(i)$ let $v_{i,\ell}\in V(\G)$ be such that $w_{i,\ell}\in G_{v_{i,\ell}}$, noting that $v_{i,\ell}\in\Star(v_i)$.
    If the prism word \[(w_{1,1}\cdots w_{1,i(1)})\cdots (w_{k,1}\cdots  w_{k,j(k)})\cdots (w_{n,1}\cdots w_{n,j(n)})\] is not geodesic, then by \Cref{lem:charac_geod} we can find $k<k'$ and $p,p'$ such that $v_{k,p}=v_{k',p'}$, and such that each letter $w_{i,\ell}$ between $w_{k,p}$ and $w_{k',p'}$ has corresponding vertex $v_{i,\ell}\in\Star(v_{k,p})$. In particular $w_{k',p'}$ commutes with all these letters. Let $x=w_{k,p}w_{k',p'}$, represented by the empty word in the case $w_{k,p} = (w_{k',p'})^{-1}$.  Using the prism geodesic criteria from \Cref{lem:charac_geod}, and the fact that $x,w_{k,p},$ and $w_{k',p'}$ are all in the same vertex group, one can verify that $g$ has star-geodesic representative
    \[s_1\cdots s_{k-1}s'_{k}s_{k+1}\cdots s_{k'-1}s'_{k'}s_{k'+1}\cdots s_n,\]
    where $s'_k\in G_{\Star(v_k)}$ and $s'_{k'}\in G_{\Star(v_{k'})}$ have respective prism-geodesic representatives
    \begin{align*}
    s'_k&=w_{k,1}\cdots w_{k,p-1}xw_{k,p+1}\cdots w_{k,j(k)}\qquad\text{and}\\
    s'_{k'}&=w_{k',1}\cdots w_{k',p'-1}w_{k',p'+1}\cdots w_{k,'j(k')}.
    \end{align*} 
    
    Thus we obtain a star geodesic for $g$ whose letters have prism geodesic representatives such that their total prism word length is strictly lower than that of the word we started with. The proof ends by applying this process finitely many times.
\end{proof}

\begin{thm} \label{thm:qie-star graph}If $H$ is a finitely generated almost join-free subgroup of a graph product $\gp$, where $\G$ has no isolated vertices, then $H$ is quasi-isometrically embedded in the star Cayley graph of $\gp$.
\end{thm}

\begin{proof}
    Recall that, by \Cref{prop:quasi-isometric embedding}, $H$ quasi-isometrically embeds into the prism Cayley graph of $\gp$. It therefore suffices to show that the star and prism metrics are quasi-isometric on $H$.
    
    We first observe that $S_p\subset S_\star$ implies $|\cdot |_\star\leq |\cdot|_p$. For the other bound, take $h \in H$, and let $s_1\cdots s_n$ be a star-geodesic representative obtained from \Cref{lem:StarAndPrismGeodesic}, so that $|h|_p=|s_1|_p+\cdots+|s_n|_p$. As $H$ is almost join-free, \Cref{thm:join-busting} ensures that $H$ is $N$--join-busting for some $N>0$. As $\G$ has no isolated vertices, every star is a join, so the prism length of each $s_i$ is at most $N$. We therefore have that $|h|_p\leq N|h|_\star$.
\end{proof}

Using \Cref{lem:Contact=star}, $(\ref{thm-part:recognizing-stable-ploxo}) \implies (\ref{thm-part:recognizing-stable-qi})$ is now a direct corollary of \Cref{thm:qie-star graph}.

\begin{cor} \label{thm:ploxo-impies-qi}
    Suppose $H$ is a finitely generated almost join-free subgroup of a graph product $\gp$, where $\G$ has no isolated vertices.  Then $H$ is quasi-isometrically embedded in the contact graph $\CX$ via its orbit maps.
\end{cor}

\subsection{Stable subgroups are almost join-free}\label{sec:stability}

In this section we prove $(\ref{thm-part:recognizing-stable-stable}) \implies (\ref{thm-part:recognizing-stable-ploxo})$ in \Cref{thm-intro:recognizing-stable}. This also follows from the more general well-known fact that any infinite intersection with a quasi-isometrically embedded direct product of infinite groups is an obstruction to a subgroup being stable. As we have not yet found a direct reference for this, we include the below for completeness.

\begin{prop}\label{prop:InfJoinNotStable}
Let $\gp$ be a graph product of finitely generated infinite groups. If $H\leqslant \gp$ is stable, then $H$ is almost join-free.
\end{prop}

\begin{proof}
Suppose $H$ is not almost join-free, so there exists infinite $A\subset H$ such that $A$ is conjugate into a join subgroup $G_{\Lambda}$.  As stability is preserved under conjugation (\Cref{lem:ConjStableIsStable}), we may assume $A$ itself is contained in $G_{\Lambda}$.

Consider a standard generating set $S$ for $\gp$ (i.e. the generating set obtained by taking a union of finite generating sets for each vertex group in $\mathcal{G}$). By the normal form for graph products \cite{greennormalform}, under this generating set every parabolic subgroup can be considered as being isometrically embedded in $\gp$. As ${\Lambda}$ is a join of two subgraphs $\G_1$ and $\G_2$ of $\Gamma$, recall that $G_{\Lambda}$ is isomorphic to $G_{\G_1}\times G_{\G_2}$, and each one of $G_{\Lambda}, G_{\G_1}$, and $G_{\G_2}$ are isometrically embedded in $\gp$ under the generating set $S$.

As $A=H\cap G_{\Lambda}$ is infinite, the projection of $A$ to at least one of $G_{\G_1}$ and $G_{\G_2}$ must also be infinite. Call these projections $X$ and $Y$ respectively, and without loss of generality assume that $X$ is infinite. For every $M>0$ we can therefore pick $x_1,x_2\in X$ and $y_1,y_2\in Y$ such that $d_{S}(x_1,x_2)>M$ and $x_1y_1,x_2y_2\in A$. It follows from \cite[Lemma 6.2]{DT} that there exist (3,0)-quasi-geodesics in $G_{\Lambda}$, and therefore in $\gp$, between $x_1y_1$ and $x_2y_2$ such that at least one of these is not in the $M$-neighborhood of the other. This means that $H$ cannot be stable.
\end{proof}

\subsection{Quasi-isometrically embedded subgroups are stable}\label{sec:qi->stable}

The following proves $(\ref{thm-part:recognizing-stable-qi}) \implies (\ref{thm-part:recognizing-stable-stable})$ in \Cref{thm-intro:recognizing-stable}. This follows as a simple consequence of \Cref{thm:valiunas}, and the work of Abbott and Manning in \cite{abbott2022acylindrically}.

\begin{prop} \label{prop:qi-ilplies-stable}
    Let $H$ be a finitely generated subgroup of a finitely generated graph product $\gp$. If the orbit maps of $H$ into the contact graph $\CX$ are quasi-isometric embeddings, then $H$ is stable.
\end{prop}

\begin{proof}
    If $H$ is finite, then it is automatically stable, so assume that $H$ is infinite. Recall from \Cref{thm:valiunas} that $\CX$ is hyperbolic, and that the action of $\gp$ on $\CX$ is acylindrical. Because the orbit of $H$ into $\CX$ is a quasi-isometric embedding by assumption, we can conclude that $(\gp, \CX, H)$ is an A/QI triple, as defined in \cite[Definition 1.2]{abbott2022acylindrically}. We also have that $\gp$ is finitely generated by assumption. Thus, \cite[Theorem 1.5]{abbott2022acylindrically} implies that $H$ is stable.
\end{proof}

%%%%%%%%%%%%%%%%%%%%%%%%%%%%%%%%%%%%%%%%%%%
\section{Purely loxodromic subgroups}
\label{sec:purely-loxodromic}
%%%%%%%%%%%%%%%%%%%%%%%%%%%%%%%%%%%%%%%%%%%

In \cite[Theorem 1.1]{RAAGstable}, it is shown that the stable subgroups of right-angled Artin groups are exactly the finitely generated purely loxodromic subgroups. We recall that an isometry $g$ of a hyperbolic space $X$ is called \textit{loxodromic} if $\langle g\rangle$ has unbounded orbits, and fixes exactly two points on the Gromov boundary $\partial X$. If $G$ is a group acting by isometries on $X$, a subgroup $H\leqslant G$ is \textit{purely loxodromic} if every infinite order $h\in H$ acts loxodromically on $X$.

We can show that every stable subgroup of a graph product is purely loxodromic via the following standard observation, combined with \Cref{thm-intro:recognizing-stable} and the fact that the action of $\gp$ on $\CX$ is acylindrical (see \Cref{thm:valiunas}). Note that in an acylindrical action on a hyperbolic space, every element with unbounded orbits is loxodromic \cite[Lemma 2.2] {Bowditch2008}.

\begin{lem}
    \label{lem:qie-implies-pl}
    Let $G$ be a group acting acylindrically on a hyperbolic space $X$. Let $H$ be a finitely generated subgroup of $G$. If the orbit of $H$ into $X$ is a quasi-isometric embedding, then $H$ is purely loxodromic with respect to that action.
\end{lem}

\begin{proof}
    Let $h\in H$ be an infinite order element, so $\langle h\rangle$ has infinite diameter in $H$. By assumption, the orbit of $H$ into $X$ is a quasi-isometric embedding, so $\langle h\rangle$ has unbounded orbits in $X$. As the action of $G$ on $X$ is acylindrical, $h$ must be loxodromic.
\end{proof}

On the other hand, if a subgroup of a graph product is purely loxodromic, then it does not in general have to be stable. If $\G$ is connected, and some vertex group $G_{v}$ is an infinite finitely generated torsion group, then $G_{v}$ is (vacuously) purely loxodromic. On the other hand, $G_v$ has infinite intersection with a join subgroup, so is not almost join-free, and therefore is not stable by \Cref{prop:InfJoinNotStable}.

However, if we add the assumption that $\gp$ cannot contain infinite torsion subgroups (for example if $\gp$ is virtually torsion-free, as in the right-angled Artin group case), we are able to obtain equivalence. The following proposition (which will help us establish the equivalence) is known, see for example \cite[Proposition 2.10]{Fioravanti2024}, however we include a short proof here for completeness.

\begin{prop}
\label{prop:join-implies-elliptic}
    Let $\gp$ be a graph product, and $P(\gp)$ the prism complex of the graph product $\gp$. Then every element in $\gp$ that is conjugate into a join subgroup has bounded orbits in the action of $\gp$ on the contact graph $\CX$.
\end{prop}

\begin{proof}
    Let $G_{\Lambda}\leqslant \gp$ be a join subgroup, and let $\G_1$ and $\G_2$ be subgraphs of $\G$ such that $G_{\Lambda}=G_{\G_1}\times G_{\G_2}$. Let $v_1\in \G_1$ and $v_2\in \G_2$. Then $\G_1\subset \Star(v_2)$ and $\G_2\subset \Star(v_1)$. Therefore every $h\in G_{\Lambda}$ can be written as $h=h_1h_2$, with $h_1\in G_{\Star(v_2)}$ and $h_2\in G_{\Star(v_1)}$.
    
    This implies that $G_{\Lambda}$ lies in a ball of radius 2 around the identity in the star Cayley graph of $\gp$. As a consequence, any $g\in G_{\Lambda}$ must have bounded orbits in the star Cayley graph, and so any conjugate must also have bounded orbits. Recall that the star Cayley graph is $\gp$--equivariantly quasi-isometric to $\CX$, which concludes the proof.
\end{proof}

\begin{cor}
    \label{cor:pl-implies-ajf}
    Let $\gp$ be a graph product with no infinite torsion subgroups. If $H\leqslant \gp$ is purely loxodromic with respect to the action on $\CX$, then $H$ is almost join-free.
\end{cor}

\begin{proof}
    Suppose $H$ is purely loxodromic but has infinite intersection with some conjugate of a join subgroup. By our assumption on $\gp$, this intersection contains an element of infinite order, which as $H$ is purely loxodromic must have unbounded orbits in the action on $\CX$. This contradicts \Cref{prop:join-implies-elliptic}.
\end{proof}

As a result, by combining \Cref{thm-intro:recognizing-stable} with \Cref{lem:qie-implies-pl} and \Cref{cor:pl-implies-ajf}, in this case we obtain a additional condition equivalent to stability.

\begin{cor}\label{thm:main-torsion-free}
    Let $\Gamma$ be a finite simple graph with no isolated vertices, $\mathcal{G}=\set{G_v}{v\in V(\Gamma)}$ be a collection of infinite finitely generated groups with no infinite torsion subgroups, and $P(\gp )$ the prism complex of the graph product $\gp$. A finitely generated subgroup $H$ of $\gp$ is stable if and only if it is purely loxodromic with respect to the action on the contact graph $\CX$.
\end{cor}

%%%%%%%%%%%%%%%%%%%%%%%%%%%%%%%%%%%%%%%%%%%%%
\section{Infinite index Morse subgroups}
\label{sec:morse}
%%%%%%%%%%%%%%%%%%%%%%%%%%%%%%%%%%%%%%%%%%%%%

A subgroup $H$ of a finitely generated group $G$ is said to be \textit{Morse} (sometimes called \textit{strongly quasiconvex}) if for any constants $\lambda \geqslant 1, c\geqslant 0$, there is a constant $M\geqslant 0$ (depending on $\lambda, c$) such that every $(\lambda, c)$–quasi-geodesic in $G$ with endpoints in $H$ is contained in the $M$--neighborhood of $H$. Morse subgroups can be related to stable subgroups via the following result.

\begin{prop}
    \label{prop:morse-hyperbolic}
    \emph{\cite[Theorem 4.8]{Tran2019}\cite{Genevois2019}}
    Let $G$ be a finitely generated group. A subgroup $H$ of $G$ is stable if and only if it is Morse and hyperbolic.
\end{prop}

In \cite{Tran2019}, it is shown that the stable subgroups of right-angled Artin groups with connected defining graphs are exactly the infinite index Morse subgroups. This was also shown for mapping class groups in \cite{HKim}. Here we extend this characterization to the stable subgroups of graph products of infinite groups, where the defining graph is connected. This connected assumption is necessary, as the factors in a free product will always be Morse and infinite index, but not necessarily hyperbolic.

We begin with the following easy observation.

\begin{lem}
    \label{lem:stable-implies-morse}
    Let $\gp$ be a graph product of finitely generated infinite groups, where $\G$ has no isolated vertices. If $H\leqslant \gp$ is stable, then $H$ is infinite index and Morse.
\end{lem}

\begin{proof}
    It follows from \Cref{prop:morse-hyperbolic} that $H$ is Morse. As $\G$ has no isolated vertices, $\gp$ contains a (quasi-isometrically embedded) direct product of infinite groups, so is not hyperbolic, so no finite index subgroup is hyperbolic. Again using \Cref{prop:morse-hyperbolic} (or \cite[Lemma 3.3]{DT}), this means that no finite index subgroup of $\gp$ is stable, which gives us the conclusion.
\end{proof}

The proof of the reverse direction of \Cref{lem:stable-implies-morse} (in the case that $\G$ is connected) uses the fact that Morse subgroups have finite height. A subgroup $H$ of a group $G$ is said to have \textit{finite height} in $G$ if there exists $k$ such that, for any $g_1,\dots,g_{k}\in G$ where the cosets $g_1H,\ldots, g_{k}H$ are distinct, the intersection of the conjugates $g_1Hg_1^{-1},\ldots,g_{k}Hg_{k}^{-1}$ is finite.

\begin{thm}
    \label{thm:finite-height}
    \emph{\cite[Theorem 4.15]{Tran2019}} If $H$ is a Morse subgroup of a finitely generated group $G$, then $H$ has finite height in $G$.
\end{thm}

The following results are adapted from the statements for right-angled Artin groups in \cite{Tran2019}. The next proposition is well known to experts, the proof is included for completeness.

\begin{prop}
\label{prop:morse-in-product}
    If $G_1$ and $G_2$ are finitely generated infinite groups, and $H\leqslant G_1\times G_2$ is a Morse subgroup, then either $H$ is finite, or $H$ has finite index in $G_1\times G_2$.
\end{prop}

\begin{proof}
    There are two cases to consider. The first case is that at least one of $H\cap G_1$ or $H\cap G_2$ is infinite. Suppose without loss of generality that $H\cap G_2$ is infinite. We want to show that $H$ has finite index in $G_1\times G_2$. For this purpose, it is enough to show that both $[G_1:H\cap G_1]$ and $[G_2:H\cap G_2]$ are finite.
    
    Suppose for the purpose of contradiction that $H\cap G_1$ does not have finite index in $G_1$. This means that for any $k\geqslant 1$, we can find $a_1,\ldots,a_k\in G_1$ such that $a_1(H\cap G_1),\ldots,a_k(H\cap G_1)$ are distinct. This implies that $a_1H,\ldots,a_kH$ are also distinct. Each $a_i$ commutes with the elements of $G_2$, so $H\cap G_2\leqslant a_i Ha_i^{-1}$, and therefore $H\cap G_2 \leq \bigcap_{i=1}^k a_iHa_i^{-1}$. As $H\cap G_2$ is infinite, $H$ does not have finite height, contradicting \Cref{thm:finite-height}.

    We therefore have that $H\cap G_1$ has finite index in $G_1$. As $G_1$ is infinite, this implies that $H\cap G_1$ is also infinite, so by the same reasoning as in the previous paragraph $H\cap G_2$ also has finite index in $G_2$. We can therefore conclude that $H$ has finite index in $G_1\times G_2$.

    The second case is that both $H\cap G_1$ and $H\cap G_2$ are finite. In this case we will show that $H$ is finite, so suppose for the purpose of contradiction that $H$ is infinite. We must therefore have that the projection of $H$ to at least one of $G_1$ or $G_2$ is infinite. Suppose without loss of generality that the projection to $G_1$ is infinite.
    
    Let $S_1$ and $S_2$ be finite generating sets for $G_1$ and $G_2$ respectively. As $H\cap G_1$ is finite, for any $b\in G_2$ we have that the set $G_1\times\{b\}$ contains only finitely many elements of $H$. It follows that, for any $M\geqslant 1$, the set $G_1\times B_{S_2}(M)$ contains only finitely many elements of $H$, where $B_{S_2}(M)$ is the ball of radius $M$ around $\id_{G_2}$ in $G_2$ with respect to $S_2$. Now let $h_1h_2\in H$ such that $h_1\in G_1$ and $h_2\in G_2$, and let $\gamma_1$ be an $S_1$--geodesic from $\id_{G_1}$ to $h_1$, and $\gamma_2$ be an $S_2$--geodesic from $\id_{G_2}$ to $h_2$. The concatenation of $\gamma_1\times \{\id_{G_2}\}$ with $\{h_1\}\times \gamma_2$ is therefore a geodesic in $G_1\times G_2$ from $\id_{G_1\times G_2}$ to $h_1h_2$, with both of these points lying in $H$.

    As the projection of $H$ to $G_1$ is infinite, we can choose $h_1$ to have arbitrarily large $S_1$--length. As $G_1\times B_{S_2}(M)$ contains only finitely many elements of $H$, say $m$ many, we choose $h_1$ such that $|h_1|_{S_1}\geqslant 3M(m+1)$. Therefore some point in $\gamma_1\times \{\id_{G_2}\}$ has distance at least $M$ from any point of $H$, which contradicts $H$ being Morse. Therefore $H$ must be finite.
\end{proof}

\begin{lem}
    \label{lem:intersections-stars}
    Let $\gp$ be a graph product of infinite groups, with $\G$ connected. For any $v\in V(\G)$, and any $g_1,g_2\in \gp$, there is a finite sequence of conjugates of star subgroups $g_1 G_{\Star(v)}g_1^{-1}=Q_0,Q_1,\ldots, Q_m=g_2G_{\Star(v)}g_2^{-1}$ such that $Q_{i-1}\cap Q_i$ is infinite for each $i\in\{1,\ldots,m\}$.
\end{lem}

\begin{proof}
    This follows the proof of Lemma 8.17 in \cite{Tran2019}, with the induction done on $n=|g_1^{-1}g_2|_{\star}$.
\end{proof}

\begin{cor}
    \label{lem:intersections-joins}
    Let $\gp$ be a graph product of infinite groups, with $\G$ connected. For any join subgraph $\Lambda$ of $\G$, and any $g_1,g_2\in \gp$, there is a finite sequence of conjugates of join subgroups $g_1 G_{\Lambda}g_1^{-1}=Q_0,Q_1,\ldots, Q_m=g_2G_{\Lambda}g_2^{-1}$ such that $Q_{i-1}\cap Q_i$ is infinite for each $i\in\{1,\ldots,m\}$.
\end{cor}

\begin{proof}
    Take any $v\in V(\Lambda)$, then $G_{\Star(v)}\cap G_{\Lambda}$ is infinite. We can therefore take $Q_1=g_1G_{\Star(v)}g_1^{-1}$ and $Q_{m-1}=g_2G_{\Star(v)}g_2^{-1}$, and as star subgroups are join subgroups, the conclusion follows from \Cref{lem:intersections-stars}. 
\end{proof}

\begin{cor}
    \label{cor:morse-implies-joinfree}
    Let $\gp$ be a graph product of finitely generated infinite groups, with $\G$ connected. If $H\leqslant \gp$ is infinite index and Morse, then $H$ is almost join-free.
\end{cor}

\begin{proof}
    This follows the proof of Proposition 8.18 in \cite{Tran2019}, using \Cref{prop:morse-in-product} in place of \cite[Lemma 8.16]{Tran2019}, \Cref{lem:intersections-joins} in place of \cite[Lemma 8.17]{Tran2019}, join subgroups in place of star subgroups, finite intersections in place of trivial ones, and infinite intersections in place of non-trivial ones.
\end{proof}

We finally combine \Cref{thm-intro:recognizing-stable}, \Cref{lem:stable-implies-morse}, and \Cref{cor:morse-implies-joinfree} to obtain the following.

\begin{cor}
    Let $\gp$ be a graph product of finitely generated infinite groups, with $\G$ a connected graph with at least two vertices. A subgroup $H$ of $\gp$ is stable if and only if it is infinite index and Morse.
\end{cor}

%%%%%%%%%%%%%%%%%%%%%%%%%%%%%%%%%%%%
% BIBLIOGRAPHY
%%%%%%%%%%%%%%%%%%%%%%%%%%%%%%%%%%%%
\bibliographystyle{amsalpha}
\bibliography{bib}

\end{document}